\documentclass{amsart}
\newif\ifdraft
\drafttrue

\usepackage{amssymb,indentfirst,xspace,bm}
\usepackage[retainorgcmds]{IEEEtrantools}
\usepackage{framed,paralist}
\usepackage[off]{pdfsync}
\usepackage{mathtools}
\usepackage[normalem]{ulem}
\usepackage[pdfborder={0 0 .1}]{hyperref}
\usepackage[initials,nobysame]{amsrefs}

\makeatletter
\def\@wraptoccontribs#1#2{}

\@mparswitchfalse
\makeatother

\ifdraft
    \newcommand{\sidenote}[1]{\pdfsyncstop\marginpar[\raggedleft\tiny #1]{\raggedright\tiny #1}\pdfsyncstart}
    \newcommand{\teqref}[1]{\emph{#1}}
\else
    \newcommand{\sidenote}[1]{}
    \newcommand{\teqref}{\error}
\fi
\newcommand{\warning}[1]{\typeout{}\typeout{WARNING: #1 at line \the\inputlineno}\typeout{}}
    {\endMakeFramed}

\makeatletter
\newcommand{\UWave}[2][blue]{\bgroup \markoverwith{\textcolor{#1}{\lower3.5\p@\hbox{\sixly \char58}}}\ULon{#2}}
\makeatother
\newcommand{\SOut}[2][red]{\bgroup\markoverwith {\textcolor{#1}{\rule[.45ex]{2pt}{.1ex}}}\ULon{#2}}

\newcommand{\highlight}[2][yellow]{\bgroup\markoverwith {\textcolor{#1}{\rule[-.2em]{2pt}{1.2em}}}\ULon{#2}}

\newenvironment{beqn}[1][:C?s]%
    {\left\{\begin{IEEEeqnarraybox}[\relax][c]{#1}}%
    {\end{IEEEeqnarraybox}\right.}

\newcommand{\del}{\partial}

\newcommand{\lap}{\Delta}

\newcommand{\inv}{^{-1}}

\newcommand{\grad}{\nabla}

\newcommand{\gradperp}{\grad^\perp}
\newcommand{\divergence}{\grad \cdot}
\newcommand{\curl}{\grad \times}

\renewcommand{\epsilon}{\varepsilon}
\renewcommand{\leq}{\leqslant}
\renewcommand{\geq}{\geqslant}

\newcommand{\R}{\mathbb{R}}

\newcommand{\T}{\mathbb{T}}
\newcommand{\E}{\mathcal{E}}

\newif\iftextstyle
\textstyletrue
\everydisplay\expandafter{\the\everydisplay\textstylefalse}

\DeclarePairedDelimiter{\abs}{\lvert}{\rvert}
\DeclarePairedDelimiter{\norm}{\lVert}{\rVert}
\DeclarePairedDelimiter{\average}{\langle}{\rangle}
\newcommand{\ip}[2]{\average{#1,#2}}
\newcommand{\altip}[2]{\average{\!\ip{#1}{#2}\!}}
\newcommand{\suchthat}{\;\iftextstyle|\else\big|\fi\;}
\newcommand{\defeq}{\stackrel{\scriptscriptstyle \text{def}}{=}}

\numberwithin{equation}{section}
\allowdisplaybreaks

\newtheorem{theorem}{Theorem}[section]

\newtheorem{lemma}[theorem]{Lemma}
\newtheorem{proposition}[theorem]{Proposition}
\newtheorem{corollary}[theorem]{Corollary}

\newtheorem*{theorem*}{Theorem}
\newtheorem*{lemma*}{Lemma}
\newtheorem*{proposition*}{Proposition}
\newtheorem*{corollary*}{Corollary}

\theoremstyle{definition}

\theoremstyle{remark}

\newtheorem*{remark*}{Remark}

\newcommand{\hdiv}{H_\text{div}}
\newcommand{\eqsysENS}{\eqref{eqnNS}--\eqref{eqnU0}\xspace}

\renewcommand{\setminus}{-}

\begin{document}

\title[Coercivity and stability for extended Navier-Stokes]{Coercivity and stability results for an extended Navier-Stokes system}
\author{Gautam Iyer}
\address{Department of Mathematical Sciences, Carnegie Mellon University, Pittsburgh PA 15213}
\email{gautam@math.cmu.edu}
\author{Robert L. Pego}
\address{Department of Mathematical Sciences, Carnegie Mellon University, Pittsburgh PA 15213}
\email{rpego@cmu.edu}
\author{Arghir Zarnescu}
\address{Department of Mathematics, University of Sussex, Pevensey III, Falmer, BN1 9QH, United Kingdom.}
\email{A.Zarnescu@sussex.ac.uk}
\dedicatory{Dedicated to Peter Constantin, on the occasion of his 60\textsuperscript{th} birthday.}
\thanks{This material is based upon work supported by the National Science Foundation under grant nos. DMS 0604420,  DMS 0905723, DMS 1007914 and partially supported by the Center for Nonlinear Analysis (CNA) under the National Science Foundation Grant no. 0635983 and PIRE Grant no. OISE-0967140. The work of AZ was partly supported by an  EPSRC Science and Innovation
award to the Oxford Centre for Nonlinear PDE (EP/E035027/1).
Part of this work was done during a visit of AZ to Carnegie Mellon
University, the Centre for Nonlinear Analysis, whose support is gratefully
acknowledged.}

\begin{abstract}
In this article we study a system of equations that is known to
{\em extend} Navier-Stokes dynamics in a well-posed manner
to velocity fields that are not necessarily divergence-free.
Our aim is to contribute to an understanding of
the role of divergence and pressure in developing
energy estimates capable of controlling the nonlinear terms.
We address questions of global existence and stability
in bounded domains with no-slip boundary conditions.
Even in two space dimensions,
global existence is open in general, and remains so,
primarily due to the lack of a self-contained $L^2$ energy estimate.
However, through use of new $H^1$ coercivity estimates for the linear equations,
we establish a number of global existence and stability results,
including results for small divergence and a time-discrete scheme. 
We also prove global existence in 2D for any initial data, 
provided sufficient divergence damping is included.
\end{abstract}
\maketitle

\section{Introduction}

The zero-divergence constraint and the associated pressure field
are the source of both difficulties and benefits in the study
of the Navier-Stokes equations for the flow of viscous incompressible fluids.
On one hand, the divergence constraint complicates analysis and
approximation in a number of ways. 
For example, it produces a well-known inf-sup compatibility condition 
for mixed approximations that makes it difficult to achieve 
high accuracy with simple kinds of discretization.
On the other hand, the incompressibility constraint is responsible for the energy inequality, an estimate which is fundamental to global existence theory.

In this article we study global existence and stability
questions for a non-de\-gen\-er\-ate parabolic system
that is known to {\em extend} Navier-Stokes dynamics in
a well-posed manner to velocity fields that are not necessarily
divergence-free.
This system appeared recently in~\cite{bblLiuLiuPego}, and begins to
explain the good performance of certain numerical schemes where the
pressure is computed by solving boundary-value problems~\cite{bblLiuLiuPego2010}.
The idea to determine pressure by solving boundary-value problems
was also a feature of an earlier analytical study by Grubb and Solonnikov~\cites{bblGrubbSolonnikov2,bblGrubbSolonnikov3}, and the system we consider is equivalent
to one of their several `reduced' models.

Explicitly, we study the initial-boundary value problem
\begin{eqnarray}
\label{eqnNS}
&\del_t u + u \cdot \grad u + \grad p - \lap u = 0 &
\qquad\text{in $\Omega$},
\\
\label{eqnPdef}
&\grad p = (I-P)( \lap u - \grad \nabla\cdot u  -u\cdot\grad u  ), &
\\
\label{eqnNoSlipBC}
&u=0 & \qquad  \text{on $\del\Omega$},
\\ \label{eqnU0}
&u=u_0 & \qquad\text{in $\Omega$, when $t=0$}.
\end{eqnarray}
Here $u = u(x,t)$ is the velocity field, $p = p(x,t)$ the pressure,
and $P$ is the standard Leray projection of $L^2(\Omega,\R^d)$
onto the subspace of divergence-free vector fields which
are tangential at the boundary.
For simplicity we have taken the kinematic viscosity to be unity and
omitted body forces.

For the system \eqsysENS,
neither the initial data nor the solution are required to be divergence free.
Equation \eqref{eqnPdef} {\em defines} the pressure gradient, and
replaces the incompressibility constraint
\begin{equation}\label{e.div1}
\nabla\cdot u = 0
\qquad\text{in $\Omega$},
\end{equation}
that appears in the standard incompressible Navier-Stokes system.
However, if initially $\divergence u_0 = 0$, then the incompressibility constraint~\eqref{e.div1} holds for all time. This follows because \eqref{eqnNS}--\eqref{eqnNoSlipBC} show that
$\divergence u$ satisfies the heat equation
with no-flux boundary conditions:
\begin{equation}\label{eqnHeat}
\left\{
\begin{aligned}
\del_t \divergence u &= \lap \divergence u &&\text{in }\Omega,\\
\frac{\del}{\del \nu} \divergence u &= 0 &&\text{for }x \in \del \Omega, t > 0,
\end{aligned}\right.
\end{equation}
where $\frac{\del}{\del \nu}$ denotes the derivative with respect
to the outward unit normal to $\del \Omega$\footnote{Note that
$\int_\Omega \divergence u_0 = \int_{\del \Omega} u_0 \cdot \nu = 0$,
and so the compatibility condition for~\eqref{eqnHeat} is satisfied.}.
Thus if $\divergence u_0 = 0$, then the system~\eqsysENS reduces to the standard incompressible Navier-Stokes equations, and in this sense we say that the system~\eqsysENS extends the dynamics of the standard incompressible Navier-Stokes equations.

Of course, the dynamics of the standard incompressible Navier-Stokes equations could alternately be extended by completely omitting the $\grad \divergence u$ term from~\eqref{eqnPdef}. However, the presence of this term is crucial to the theory for two reasons. First, the~$\lap \divergence u$ term in~\eqref{eqnHeat} is a direct result of the $\grad \divergence u$ term in~\eqref{eqnPdef}, and provides exponential stability of the divergence free subspace. Thus, from a numerical perspective, errors in the divergence will be exponentially damped. The second, and perhaps deeper reason, is the essential role played by $\grad \divergence u$ in the well-posedness results of~\cites{bblLiuLiuPego,bblGrubbSolonnikov2,bblGrubbSolonnikov3}. We elaborate on this below.

Define the {\em Stokes pressure gradient}, $\grad p_s(u)$, by
\begin{equation}\label{eqnPsDef}
\grad p_s(u) = (I-P)(\lap u - \grad \nabla\cdot u).
\end{equation}
In context we often use $p_s$ to denote $p_s(u)$.
Given any $u\in H^2(\Omega,\R^d)$, the function $p_s(u)$
is determined as the unique mean-zero solution
to the boundary-value problem
\begin{equation}\label{eqnPsBVP}
\lap p_s=0 \quad\text{in $\Omega$},
\qquad
\nu\cdot\grad p_s = \nu\cdot(\lap -\grad\divergence )u
\quad\text{on $\del\Omega$}.
\end{equation}
Without the $\grad \divergence u$ term, the boundary condition in~\eqref{eqnPsBVP} would not make sense for all $u \in H^2(\Omega)$. With the $\grad \divergence u$ term, however, $\lap u-\grad\divergence u$ is $L^2$ and divergence-free; hence a standard trace theorem~\cite{bblConstFoias}*{Proposition 1.4} makes sense of the boundary condition in $H^{-1/2}(\del\Omega)$. The Grubb-Solonnikov~\cites{bblGrubbSolonnikov2,bblGrubbSolonnikov3} approach is based on using the boundary-value problem \eqref{eqnPsBVP} to determine the contribution of $\grad p_s(u)$ to $\grad p$, and proves well-posedness of~\eqsysENS using a theory of parabolic pseudo-differential initial-boundary value problems in $L^p$-based Sobolev spaces.

The convergence arguments of Liu et al.~\cite{bblLiuLiuPego}, on the other hand, result in
a comparatively simple local well-posedness proof
for \eqsysENS
for initial velocity in $H^1_0(\Omega)$.
This proof is based instead on the expression of the Stokes pressure gradient as a Laplace-Leray commutator:
\begin{equation}\label{eqnPCommutator}
\grad p_s = (\lap P-P\lap)u.
\end{equation}
This follows directly from \eqref{eqnPsDef} using the fact that
\begin{equation}\label{eqnGradDivId}
\grad\divergence u=\lap (I-P)u.
\end{equation}
Even in this approach, the `extra' $\grad \divergence u$ term in~\eqref{eqnPsDef} is directly responsible for the commutator representation~\eqref{eqnPCommutator}. The key idea used in~\cite{bblLiuLiuPego} is to treat the Stokes pressure gradient as the Laplace-Leray commutator~\eqref{eqnPCommutator}, and show (Theorem~\ref{thmCommutatorEstimate}, below) that it is dominated by the Laplacian (cf.~\eqref{eqnGradPsBound}) to leading order.

While the methods of~\cites{bblGrubbSolonnikov2,bblGrubbSolonnikov3,bblLiuLiuPego} effectively address local well-posedness of~\eqsysENS, they do not address global existence or stability. For the standard incompressible Navier-Stokes equations in three space dimensions, global existence of strong solutions is a well-known fundamental open problem~\cites{bblFeffermanClay,bblConstOpenProblems}. However, classical results establish global existence and regularity if if the flow is two-dimensional~\cite{bblLeray}, or the initial data is suitably small~\cites{bblLeray,bblConstFoias}.

In this paper, we establish a few such global existence results for the system~\eqsysENS. The main difficulty in proving a small-data global existence result for~\eqsysENS is not the nonlinearity. The root of the problem is that the \emph{linear} terms are not coercive under the standard $L^2$ inner product. We remedy this difficulty by using the commutator estimate in~\cite{bblLiuLiuPego} to construct an adjusted inner product under which the linear terms are coercive. This allows us to establish global existence for small initial data in two or three dimensions, and unconditional global stability of a time discrete scheme for the linear equations. This leads to an improved understanding of how the divergence and pressure can be handled to obtain energy estimates capable of controlling the nonlinear terms.

In 2D, we can extend our small-data global existence results to initial data with small divergence. For arbitrary initial data, we can add a sufficiently large divergence damping term to~\eqsysENS to obtain global existence. However, presently we are not able to prove global existence for~\eqsysENS for arbitrary initial data. The difficulty is that for the energy balance using the standard $L^2$ inner-product, the non-linear term is skew-symmetric, and does not contribute; however, the linear terms are not coercive. On the other hand, for the energy balance using the adjusted inner products we consider, the linear terms are coercive; however, nonlinearity is no longer skew symmetric, and contributes non-trivially.

Coercivity of the linear terms (albeit under a non-standard inner product) allows one to treat~\eqsysENS as a non-degenerate parabolic system. While this has helped simplify existence theory and the analysis of certain numerical approximation schemes, some other questions apparently become more difficult. In particular, while global existence of the standard incompressible Navier-Stokes equations is well known in 2D, it remains open for~\eqsysENS for general (2D) initial data.

\section{Main results}
\subsection{Coercivity of the extended Stokes operator.}
In the study of parabolic problems, an extremely useful (and often crucial) property is coercivity of the underlying linear operator. For~\eqref{eqnNS}, the linear operator in question is the \emph{extended Stokes operator}, $A$, defined by
\begin{equation}\label{eqnAdef}
A u \defeq - \lap u + \grad p_s(u) = -P\lap u - \grad \divergence u.
\end{equation}
Note that the last equality follows from the identity~\eqref{eqnGradDivId}. Under periodic boundary conditions, the extended Stokes operator $A$ is coercive. Indeed, under periodic boundary conditions, $P \lap = \lap P$, and so
\begin{equation}\label{eqnPeriodicStokesCoercivity}
\ip{u}{Au} = \norm{\grad u}_{L^2}^2,
\end{equation}
where $\ip{\cdot}{\cdot}$ denotes the standard $L^2$ inner product on the torus.

Under no-slip ($0$-Dirichlet) boundary conditions, the situation is surprisingly more complicated. The extended Stokes operator \emph{fails} to be positive, let alone coercive, under the standard $L^2$ inner product. To briefly explain why, observe
that for $u\in H^2\cap H^1_0(\Omega,\R^d)$,
\begin{equation}\label{eqnIpUAu}
\ip{u}{Au} = \int_\Omega u \cdot A u
    = \int_\Omega \abs{\grad u}^2 + \int_\Omega u \cdot \grad p_s.
\end{equation}
Now if $\divergence u \ne 0$, the second term on the right need not vanish.
In view of the commutator relation~\eqref{eqnPCommutator}, one might
expect $\norm{\grad p_s}_{L^2}$ to be dominated by $\norm{\grad u}_{L^2}$. This, however, is known to be false, and control of the Stokes pressure $p_s$ requires \emph{more} than one derivative on $u$. Consequently, if $\divergence u \neq 0$, then the second term on the right of~\eqref{eqnIpUAu} can dominate the first, and destroy positivity of $A$.

Since our primary interest in the extended Stokes operator is to
study~\eqsysENS, and the divergence of solutions to~\eqsysENS is well
controlled, one may hope to rectify non-positivity of $A$ by a coercivity estimate of the form
\begin{equation}\label{eqnHopefulCoercivicity}
\ip{u}{Au} \geq \epsilon \norm{\grad u}_{L^2}^2 - C \norm{\divergence u}_{L^2}^2.
\end{equation}
But again, this turns out to be false.
\begin{proposition}[Failure of Coercivity]\label{ppnNonPositivity}
Let $\Omega \subset \R^2$ be a bounded, simply connected $C^3$ domain. For any $\epsilon, C \geq 0$, there exists a function $u \in C^2(\bar\Omega)$ such that
\begin{equation}\label{eqnNonPositivity}
u = 0 \text{ on }\del \Omega,
\qquad\text{and}\qquad
\ip{u}{Au} \leq \epsilon \norm{\grad u}_{L^2}^2 - 
C \norm{\divergence u}_{L^2}^2.
\end{equation}
\end{proposition}
The key idea in the proof is to identify the \emph{harmonic conjugate of the Stokes pressure as the harmonic extension of the vorticity}. Since this is independent of our main focus, we present the proof of Proposition~\ref{ppnNonPositivity} in Appendix~\ref{sxnNonPositivity}, towards the end of this paper. We remark, however, that if $u \in H^2\cap H$, where
$$
H \defeq \{ v \in L^2(\Omega) \suchthat v = Pv \} = \{v \in L^2(\Omega) \suchthat \divergence v = 0, \text{ and } v \cdot \nu = 0 \text{ on } \del \Omega \},
$$
then the second equality in~\eqref{eqnAdef} shows that the extended
Stokes operator $A$ reduces to the standard Stokes operator $-P\lap$. 
In the space
$H^2\cap H^1_0 \cap H$
coercivity of the standard Stokes operator is well known. 
Namely
\eqref{eqnPeriodicStokesCoercivity} holds
for all $u \in H^2\cap H^1_0 \cap H$ 
(see for instance~\cite{bblConstFoias}*{Chapter 4}). 
Unfortunately, when we
consider vector fields for which $u \notin H$, Proposition~\ref{ppnNonPositivity} shows that coercivity fails for the extended Stokes operator.

The key to global existence results for the nonlinear system~\eqsysENS is to remedy the negative results in Proposition~\ref{ppnNonPositivity}
in a manner that interacts well with the nonlinear term. This can be
done by introducing a stabilizing higher order term, and a compensating
gradient projection term, as we now describe.

For any $u \in H^1(\Omega)$ define $Q(u)$, the primitive of the gradient projection, to be the unique mean zero $H^1$ function such that
$$
\grad Q(u) = (I - P) u.
$$
Given constants $\epsilon, C > 0$, we define an $H^1$-equivalent inner product $\altip{\cdot}{\cdot}_{\epsilon, C}$ by
\begin{equation}\label{d:altip}
\altip{u}{v}_{\epsilon, C} = \ip{u}{v} + \epsilon \ip{\grad u}{\grad v} + C \ip{Q(u)}{Q(v)},
\end{equation}
where $\ip{\cdot}{\cdot}$ denotes the standard inner product on $L^2(\Omega)$. Our main result shows that for all $\epsilon$ sufficiently small, we can find $C$ large enough to ensure coercivity of $A$ under the inner product~$\altip{\cdot}{\cdot}_{\epsilon, C}$.
\begin{proposition}[$H^1$-equivalent coercivity]\label{ppnAdjustedIP}
Let $\Omega\subset\R^d$ be a $C^3$ domain. There exists positive constants $\epsilon_0 = \epsilon_0(\Omega)$ and $c = c(\Omega)$ such that for any $\epsilon \in (0, \epsilon_0)$, there exists a constant $C_\epsilon = C_\epsilon(\Omega) > 0$, such that for the inner product $\ip{\cdot}{\cdot}_\epsilon$ defined by
$$
\ip{\cdot}{\cdot}_\epsilon \defeq \altip{\cdot}{\cdot}_{\epsilon, C_\epsilon},
$$
we have
\begin{equation}\label{eqnCoercivicity}
\ip{u}{Au}_{\epsilon} \geq \frac{1}{c} \left( \norm{\grad u}_{L^2}^2 + \epsilon \norm{\lap u}_{L^2}^2 + C_\epsilon \norm{\grad q}_{L^2}^2\right)
\end{equation}
for all $u \in H^2 \cap H^1_0$. Consequently, there exists a constant $C_\epsilon' = C_\epsilon'(\epsilon, \Omega)$ such that
\begin{equation}\label{eqnCoercivicity2}
\ip{u}{Au}_\epsilon \geq \frac{1}{c} \ip{u}{u}_\epsilon,
\quad\text{and}\quad
\ip{u}{Au}_\epsilon \geq \frac{1}{C_\epsilon'} \ip{\grad u}{\grad u}_\epsilon,
\end{equation}
for all $u \in H^2 \cap H^1_0$.
\end{proposition}
We prove this Proposition in Section~\ref{sxnAdjustedIP}. 
The main ingredient in the proof is an estimate 
for the Laplace-Leray commutator \eqref{eqnPCommutator}
that is proved in~\cite{bblLiuLiuPego} 
and stated in Theorem~\ref{thmCommutatorEstimate} below. 
A couple of further consequences of this Theorem are worth mentioning
here.  First, $A$ is invertible on $L^2$ with compact resolvent
(Lemma~\ref{lmaAinv}).  And, due to
Theorem~\ref{thmCommutatorEstimate} and the self-adjointness of the
Laplacian, an elementary result about sectorial
operators~\cite{bblHenry}*{Theorem 1.3.2} directly implies that $A$ is a
sectorial operator on $L^2$ with domain $D(A)=D(-\lap)=H^2\cap H^1_0$. 

The result of Proposition~\ref{ppnAdjustedIP} raises the question of
whether coercivity of $A$ can be obtained in a space with less
regularity than $H^1$ by using an equivalent inner product.  
In this regard we have two remarks.
First, in Proposition~\ref{ppnAdjustedIPPrime} we will describe an inner product
$\ip{\cdot}{\cdot}_\epsilon'$ for which $A$ is coercive
that is equivalent to the usual inner product on 
the space
$$
\hdiv = \{v \in L^2(\Omega) \suchthat 
\divergence v \in L^2 \text{ and } v \cdot \nu = 0 \text{ on } \del \Omega \}.
$$
Second, we expect that a bilinear form defined by
\begin{equation}
\ip{u}{v}''_\epsilon \defeq \ip{ A^{-1/2} u}{ A^{-1/2} v}_\epsilon
\end{equation}
determines an $L^2$-equivalent inner product under which $A$ is
coercive. Coercivity for $u\in D(A)$ would follow from 
Proposition~\ref{ppnAdjustedIP}, and $L^2$ continuity by well-known
interpolation estimates. However, an $L^2$-coercivity bound 
$\ip{u}{u}_\epsilon''\geq c\norm{u}_{L^2}^2$ appears not to be
easy to prove --- it may involve proving $A$ has bounded
imaginary powers (see \cite{bblAbels}) in order to establish the expected
characterization $D(A^{1/2})=H^1_0$.

In any case, unfortunately the
inner products $\ip{\cdot}{\cdot}'_\epsilon$ and
$\ip{\cdot}{\cdot}''_\epsilon$ do not seem to interact well with the
nonlinearity in~\eqref{eqnNS}. Thus for questions of global existence
and stability for the nonlinear extended Navier-Stokes equations and
their discretizations, it is more convenient to use the inner product in
Proposition~\ref{ppnAdjustedIP}. The rest of the paper can be read
independently of Proposition~\ref{ppnAdjustedIPPrime} or its proof.

\subsection{Energy decay for the extended Stokes equations.}
A first step to global existence results for~\eqsysENS, is the study of long time behaviour for the underlying linear equations. These are the extended Stokes equations:
\begin{equation}\label{eqnExtendedStokes}
\left\{
\begin{aligned}
\del_t u - \lap u + \grad p_s(u) &= 0 &&\text{in }\Omega,\\
u(x,t) &= 0 &&\text{for }x \in \del \Omega, t > 0,\\
u(x, 0) &= u_0(x) && \text{for } x\in \Omega.
\end{aligned}\right.
\end{equation}
A direct consequence of Proposition~\ref{ppnNonPositivity} is that the energy of solutions to~\eqref{eqnExtendedStokes} can increase, at least initially.
\begin{corollary}\label{clyL2Increase}
There exists $u_0 \in C^2(\Omega)$ with $u_0 = 0$ on $\del \Omega$, and $t_0 > 0$ such that the solution $u$ to~\eqref{eqnExtendedStokes} with initial data $u_0$ satisfies
$$
\norm{u(t_0)}_{L^2} > \norm{u_0}_{L^2}.
$$
\end{corollary}
The proof of Corollary~\ref{clyL2Increase} can be found at the end of Appendix~\ref{sxnNonPositivity}, following the proof of Proposition~\ref{ppnAdjustedIP}.

In contrast to the extended Stokes equations, solutions to the standard
Stokes equations (with initial data in $H$) always have monotonically decaying $L^2$ norm. This follows because if $u(t) \in H$, then multiplication by $u$ and integration by parts produces the standard energy inequality
\begin{equation}\label{eqnEnergyEqualityStokes}
\frac{1}{2} \del_t \norm{u(t)}_{L^2}^2 + \norm{\grad u}_{L^2}^2 = 0.
\end{equation}
The Poincar\'e inequality now yields strict exponential decay
\begin{equation}\label{eqnStokesExpDecay}
\norm{u(t)}_{L^2}^2 \leq e^{-c t} \norm{u_0}_{L^2}^2.
\end{equation}
for all solutions to the standard Stokes equations with initial data in $H$.

Despite the counter-intuitive initial energy increase, the extended Stokes system is a well-posed, \emph{non-degenerate} parabolic system. This was proved in~\cites{bblGrubbSolonnikov2,bblLiuLiuPego}, and is a direct consequence of Theorem~\ref{thmCommutatorEstimate}.
Indeed, since $A$ is sectorial it generates an analytic semigroup $e^{-At}$, showing well-posedness of the initial-boundary-value problem~\eqref{eqnExtendedStokes}.  Because no eigenvalue of $A$ has non-positive real part
by Proposition~\ref{ppnAdjustedIP},
one can quickly show that while the $L^2$ energy of
solutions to~\eqref{eqnExtendedStokes} can increase initially, it must
eventually decay exponentially. Explicitly, this means that solutions
to~\eqref{eqnExtendedStokes} must satisfy
\begin{equation}\label{eqnExtendedStokesExpDecay}
\norm{u(t)}_{L^2}^2 \leq C e^{-c t} \norm{u_0}_{L^2}^2
\end{equation}
for some constants $C, c > 0$.

To digress briefly, we remark that with a little work, one can
explicitly characterize the spectrum of $A$. Indeed, if $A_S$ denotes the (standard) Stokes operator with no-slip boundary conditions, and $\lap_N$ denotes the Laplace operator with homogeneous Neumann boundary conditions, then
$$
\sigma(A) = \sigma(A_S) \cup \sigma(-\lap_N) \setminus \{0\}.
$$
Seeing $\sigma(A)$ is contained in the right hand side above is immediate. The reverse inclusion requires a little work, and was communicated to us by Kelliher~\cite{bblKelliher}.

Unfortunately, an abstract spectral-theoretic proof of~\eqref{eqnExtendedStokesExpDecay} is not of direct help for studying the stability of time-discrete schemes, which was a primary motivation for introducing these equations. Further, \eqref{eqnExtendedStokesExpDecay} does {not} recover~\eqref{eqnEnergyEqualityStokes} for solutions with initial data in $H$. For this reason, we search for a direct energy-method proof of~\eqref{eqnExtendedStokesExpDecay}, and for an idea which also allows the study of time discrete schemes.

Observe first that if we multiply~\eqref{eqnExtendedStokes} by~$u$,
integrate, use the commutator estimate~\eqref{eqnGradPsBound} and
Gronwall's lemma, we obtain exponential \emph{growth}, not decay, of
$\norm{u}_{L^2}^2$. If we involve a higher derivative, coercivity of $A$
in Proposition~\ref{ppnAdjustedIP} (or
Proposition~\ref{ppnAdjustedIPPrime}) and Gronwall's lemma guarantee
eventual exponential decay of $\norm{u}_{H^{1}}$ (or
$\norm{u}_{\hdiv}$). However, for~\eqref{eqnExtendedStokes}, we can
obtain a more satisfactory decay estimate by considering non-quadratic form energies.
\begin{proposition}\label{ppnStokesEnergyDecay2}
Let $u$ be a solution to~\eqref{eqnExtendedStokes} with $u_0 \in H^1(\Omega)$. Then for any $\epsilon > 0$, there exists constants $c_1 = c_1(\Omega)$ and $c_2 = c_2( \Omega, \epsilon)$, such that $c_1, c_2 > 0$ and
\begin{equation}\label{eqnNonLinearEnergyDecay}
\del_t \E_{c_1, c_2}(u) + \mathcal E'_\epsilon(u) \leq 0
\end{equation}
where $\mathcal E_{c_1, c_2}$ and $\mathcal E'_\epsilon$ are defined by
\begin{gather}
\label{eqnNLEdef} \E_{c_1, c_2}(u) \defeq \norm{u}_{L^2}^2 + c_1 \norm{\grad u}_{L^2} \norm{\grad Q(u)}_{L^2} + c_2 \norm{\grad Q(u)}_{L^2}^2,\\
\label{eqnNLEPrimeDef} \mathcal E'_\epsilon(u) \defeq (2 - \epsilon) \norm{\grad u}_{L^2}^2 + \norm{\lap u}_{L^2} \norm{\grad Q(u)}_{L^2} + \norm{\lap Q(u)}_{L^2}^2.
\end{gather}
\end{proposition}
The proof of Proposition~\ref{ppnStokesEnergyDecay2} is in Section~\ref{sxnExtStokesExpDecay}. While~\eqref{eqnNonLinearEnergyDecay} does not imply eventual exponential decay controlled only by the $L^2$ norm as in~\eqref{eqnExtendedStokesExpDecay}, it does provide an estimate that reduces to the energy inequality for extended Stokes equations~\eqref{eqnEnergyEqualityStokes} when the initial data is in $H$. To see this, note that if $u_0 \in  H$, then $(I-P) u(t) = 0$ for all $t > 0$ because $\divergence u$ satisfies the heat equation~\eqref{eqnHeat}. Consequently $Q(u) \equiv 0$, and equation~\eqref{eqnNonLinearEnergyDecay} reduces to
$$
\del_t \norm{u}_{L^2}^2 + (2 - \epsilon) \norm{\grad u}_{L^2}^2 \leq 0.
$$
Thus in the limit $\epsilon \to 0$, we naturally recover the energy decay for the Stokes equation (equation~\eqref{eqnEnergyEqualityStokes}) for initial data in $H$.

We also notice that the `energy' $\mathcal E_{c_1, c_2}$ of solutions must in fact decrease exponentially. This is because $\int_\Omega \grad u = 0 = \int_\Omega Q(u)$, and so the Poincar\'e inequality can be applied to both the terms $\norm{\grad u}_{L^2}$ and $\norm{Q(u)}_{L^2}$. Thus equation~\eqref{eqnNonLinearEnergyDecay} immediately implies
$$
\mathcal E_{c_1,c_2}(u(t)) \leq e^{-c t} \mathcal E_{c_1, c_2}(u_0),
$$
for some small constant $c = c(c_1, c_2, \epsilon, \Omega)$. Unfortunately, however, for the extended Navier-Stokes equations, the `energy' $\mathcal E_{c_1, c_2}$ does not interact well with the nonlinearity.

\subsection{Uniform stability for a time-discrete scheme.}
Before moving on to the non-linear system~\eqsysENS, we study stability
of a time-discrete scheme for~\eqref{eqnExtendedStokes}, 
of the type treated in \cite{bblLiuLiuPego}. One main
motivation for studying the system~\eqsysENS, or
the linear system~\eqref{eqnExtendedStokes},  is that this kind of time-discrete
scheme is naturally implicit only in the viscosity term, and explicit in
the pressure. We will show that the ideas used in the proof of
Proposition~\ref{ppnAdjustedIP} give \emph{globally uniform} stability
estimates for such time-discrete schemes.

Given an approximation $u^n$ to the velocity at time $n\, \delta t$, we determine $\grad p^n$ from the weak-form Poisson equation
\begin{equation}\label{disceq:weakpressure}
\ip{\grad p^n}{\grad \varphi}=\ip{\lap u^n-\grad\divergence u^n+f^n}{\grad\varphi} \qquad \forall\varphi\in H^1(\Omega).
\end{equation}
Now, we determine $u^{n+1}$ by solving the elliptic boundary value problem
\begin{equation}\label{disceq:weakStokes}
\begin{beqn}
\frac{u^{n+1}-u^n}{\delta t}-\lap u^{n+1}+\grad p^n =f^n & in $\Omega$,\\
u^{n+1} = 0 & on $\del \Omega$.
\end{beqn}
\end{equation}
where $f^n=\frac{1}{\delta t}\int_{n\delta t}^{(n+1)\delta t} f(s)\,ds$
is a time-discretized forcing term. 

\begin{proposition}\label{prop:disc}
Let $\Omega$ be a bounded domain in $\R^d$, $d=2,3$, with $C^3$ boundary.
Then there exist positive constants
$\kappa_0$, $\epsilon, C_\epsilon, C, C'$, depending only on $\Omega$,
such that whenever $0<\delta t<\kappa_0$,
then for all $N>0$ we have
\begin{multline}\label{eqnDiscEnergyIneq}
\norm{u^N}_{L^2}^2+ \epsilon \norm{\grad u^N}_{L^2}^2 
+C_\epsilon \norm{\grad q^N}_{L^2}^2
\mathop+\\
+\frac{1}{C} \sum_{k=0}^N \left( 
\norm{\grad u^k}_{L^2}^2 + \epsilon \norm{\lap u^k}_{L^2}^2
+ C_\epsilon \norm{\lap q^{k+1}}_{L^2}^2
\right)\delta t\\
\leq 
\norm{u^0}_{L^2}^2+ \epsilon \norm{\grad u^0}_{L^2}^2
+ C_\epsilon \norm{\grad q^0}_{L^2}^2
\\
+C\,\delta t \left( \norm{\grad u^0}_{L^2}^2 + \epsilon \norm{\lap u^0}_{L^2}^2 + \sum_{k=0}^N \|f^k\|_{L^2}^2 \right).
\end{multline}
and
\begin{multline}
\label{eqnDiscExpDecay}
\norm{u^N}_{L^2}^2+ \epsilon \norm{\grad u^N}_{L^2}^2 
+ C_\epsilon \norm{\grad q^N}_{L^2}^2
\\
\leq (1-C\delta t)^N\left( 
\norm{u^0}_{L^2}^2 + \epsilon \norm{\grad u^0}_{L^2}^2
+ C_\epsilon \norm{\grad q^0}_{L^2}^2 
\right) \\
 + C' \sum_{k=0}^{N-1}\norm{f_k}_{L^2}^2(1-C\delta t)^{N-1-k}\delta t .
\end{multline}
\end{proposition}

The proof of this proposition is in Section~\ref{sxnTimeDiscrete}.

\subsection{Global existence results for the extended Navier-Stokes equations.}
When one seeks an $L^2$ energy estimate for~\eqsysENS,
multiplying \eqref{eqnNS} by $u$, the nonlinearity produces the term
\begin{equation}\label{eqnExtraHarmlessTerm}
\int_\Omega u \cdot (u \cdot \grad) u
= -\frac{1}{2} \int_\Omega (\divergence u) \abs{u}^2 .
\end{equation}
In general this is non-zero, but
is morally harmless since $\divergence u$
is a solution of \eqref{eqnHeat} and is well controlled. This is indeed
the case in two dimensions, but under {\it periodic} boundary conditions
(see Proposition~\ref{ppnNLBCGexist}, and the remark following it).
The key ingredient for proving global existence for periodic boundary
conditions is the coercivity~\eqref{eqnPeriodicStokesCoercivity} of the linear terms.
Consequently, despite the extra non-linear term arising
from~\eqref{eqnExtraHarmlessTerm}, the $L^2$ energy balance closes and
the well-known existence results for the standard incompressible
Navier-Stokes equations continue to hold with minor modifications.

The situation is more complicated under no-slip boundary
conditions, however, since now 
coercivity~\eqref{eqnPeriodicStokesCoercivity} \emph{fails}.
 To get any mileage from the linear terms, we need to use an inner-product under which the linear terms are coercive. Using the inner product in Proposition~\ref{ppnAdjustedIP}, and a `brutal' estimate on the nonlinearity, we can obtain a two or three dimensional small-data global existence result.

\begin{theorem}[Small data global existence]\label{thmSmallDataGexist}
Let $d = 2$ or $3$, $\Omega \subset \R^d$ be a bounded domain with $C^3$ boundary. There exists a small constant $V_0 = V_0(\Omega) > 0$ such that if $u_0 \in H^1_0(\Omega)$ with
$$
\norm{u_0}_{H^1} < V_0
$$
then there exists a global strong solution to~\eqsysENS with
\begin{gather}\label{eqnUspaces}
u \in L^2(0, T; H^2(\Omega)\cap H^1_0(\Omega)) \cap H^1(0, T; L^2(\Omega) ).
\end{gather}
for any $T > 0$. Consequently $u \in C( [0, \infty); H^1_0)$ and $\divergence u \in C^\infty( (0, \infty) \times \Omega )$.
\end{theorem}

The proof of this theorem is in Section~\ref{sxnSmallData}. Two-dimensional global existence, however, poses a different problem. A key ingredient in 2D global existence for the standard incompressible Navier-Stokes equations is the $L^2$ energy balance: the nonlinearity cancels, and doesn't contribute! Unfortunately, for~\eqsysENS, the $L^2$-energy balance doesn't close because of the higher order contribution from the Stokes pressure gradient.

In the absence of an $L^2$ energy inequality, we are only able to prove a perturbative result. If the initial data is divergence free, then~\eqsysENS reduces to the standard incompressible Navier-Stokes equations, for which 2D global existence is well known. Thus for initial data with small divergence, we can prove 2D global existence for~\eqsysENS.

\begin{theorem}[Small divergence global existence in 2D]\label{thmSmallDivGexist}
Let $\Omega \subset \R^2$ be a bounded $C^3$ domain, $v_0 \in  H^1_0(\Omega)$ with $\divergence v_0 = 0$ be arbitrary.
There exists a small constant
$U_0 = U_0(\Omega, \norm{v_0}_{H^1_0(\Omega)}) > 0$ such that if
\begin{equation}\label{eqnSmallDivGexistAssumption}
u_0 \in H^1_0(\Omega), \quad P_0 u_0 = v_0 \quad\text{and}\quad \norm{\divergence u_0}_{L^2(\Omega)} < U_0
\end{equation}
then there exists a global strong solution to~\eqsysENS with initial data $u_0$ such that~\eqref{eqnUspaces} holds for all $T > 0$.
\end{theorem}

The operator $P_0$ above is the $H^1_0$-orthogonal projection of $H^1_0(\Omega)$ onto the subspace of divergence free vector fields, and is described in Section~\ref{sxn2DSmallDivGlobalExistence} along with the proof of Theorem~\ref{thmSmallDivGexist}. One strategy to avoid the small divergence assumption is to further damp the divergence.
Namely, for arbitrary initial data (in 2D), if we add a strong enough
divergence-damping term to~\eqref{eqnNS}--\eqref{eqnPdef}, we can guarantee global existence.

\begin{corollary}[Divergence-damped global existence in 2D]\label{clyDivergenceDampedExistence}
Let $\Omega \subset \R^2$ be a $C^3$, bounded domain and $u_0 \in  H^1_0(\Omega)$ be arbitrary. There exists a constant $\alpha_0 = \alpha_0(\Omega, \norm{\divergence u_0}_{L^2(\Omega)}) > 0$  such that if $\alpha\geq\alpha_0$ then the system
\begin{equation}\label{eqnExtendedDynamicsAlpha}
\left\{
\begin{aligned}
\del_t u + P((u \cdot \grad) u)  + Au + \alpha (I - P) u  &= 0 &&\text{in }\Omega\\
u(x,t) &= 0 &&\text{for }x \in \del \Omega, t > 0\\
u(x,0)&=u_0(x),
\end{aligned}\right.
\end{equation}
has  a global strong solution $u$ such that~\eqref{eqnUspaces} holds for all $T > 0$.
\end{corollary}

The main idea in proving Corollary~\ref{clyDivergenceDampedExistence} is to verify that the divergence-damped extended Stokes operator $B_\alpha$ defined by
\begin{equation}\label{eqnBAlphaDef}
B_\alpha \defeq A + \alpha(I-P)
\end{equation}
is coercive, with coercivity constant \emph{independent} of $\alpha$. Consequently, the proofs of Theorems~\ref{thmSmallDataGexist} and~\ref{thmSmallDivGexist} work verbatim for the system~\eqref{eqnExtendedDynamicsAlpha}, with constants independent of~$\alpha$. Combining these existence theorems, and using the added divergence damping gives Corollary~\ref{clyDivergenceDampedExistence}, a better existence result as an easy corollary. We devote Section~\ref{sxnDivergenceDamping} to the coercivity of $B_\alpha$ (Proposition~\ref{ppnBCoercive}), and the proof of Corollary~\ref{clyDivergenceDampedExistence}.\medskip

So far, our two-dimensional global existence results under no-slip boundary conditions required either a small initial divergence assumption, or an additional strong divergence damping term. 
Such requirements are not needed under periodic boundary conditions, primarily because of~\eqref{eqnPeriodicStokesCoercivity}.
We observe, then, that the identity~\eqref{eqnPeriodicStokesCoercivity} will still hold in domains with boundary, 
provided we consider functions $u$ with boundary conditions
\begin{equation}\label{eqnNlBC}
Pu \cdot \tau = 0 \text{ on } \del \Omega
\qquad\text{and}\qquad
u \cdot \nu = 0 \text{ on } \del \Omega,
\end{equation}
where $\nu$ and $\tau$ are the unit normal and tangential vectors respectively.
These boundary conditions~\eqref{eqnNlBC} reduce to the usual no-slip conditions in the physically relevant situation where $u = Pu$.

Armed with~\eqref{eqnPeriodicStokesCoercivity}, we obtain a 2D global existence result without a smallness assumption, or any additional divergence damping.
\begin{proposition}\label{ppnNLBCGexist}
Let $\Omega \subset \R^2$ be locally Lipschitz and bounded, and let $u_0 \in H^1(\Omega)$. There exists a time-global strong solution to~\eqref{eqnNS}--\eqref{eqnPdef} with initial data $u_0$ and boundary conditions~\eqref{eqnNlBC}.
\end{proposition}
We prove the identity~\eqref{eqnPeriodicStokesCoercivity} and Proposition~\ref{ppnNLBCGexist} in Section~\ref{sxnBC}. The proof of Proposition~\ref{ppnNLBCGexist} emphasizes another (analytical) advantage of the boundary conditions~\eqref{eqnNlBC}. Under all the boundary conditions we consider (no-slip, periodic, and~\eqref{eqnNlBC}) the evolution equation for the gradient projection is always linear, self contained, and decays at an explicitly known rate. The evolution equation for the Leray projection (equation~\eqref{per.v}), is coupled to the gradient projection; however the coupling terms are harmless. What causes trouble under the no-slip boundary conditions is that the evolution of the Leray projection is also coupled to the gradient projection through boundary conditions! This proves problematic in the case of 2D global existence. On the other hand, periodic boundary conditions, or the boundary conditions~\eqref{eqnNlBC} provide an explicit \emph{de-coupled} boundary condition for the Leray projection, which simplifies the analysis greatly. Unfortunately, the price paid is that the boundary conditions~\eqref{eqnNlBC} are much harder to implement numerically.

\section{Coercivity of the extended Stokes operator.}\label{sxnAdjustedIP}
As mentioned earlier, the extended Stokes operator is \emph{not} coercive under the standard $L^2$ inner product. However, it \emph{is} coercive under a non-standard, but $H^1$-equivalent, inner product. This is the main tool we use in studying the extended Navier-Stokes. The aim of this section is to prove Proposition~\ref{ppnAdjustedIP} (coercivity under the adjusted $H^1$ inner product).
The main ingredient in the proof is the following estimate on the 
Laplace-Leray commutator.

\begin{theorem}[Liu, Liu, Pego~\cite{bblLiuLiuPego}]\label{thmCommutatorEstimate}
Let $\Omega$ be a connected, bounded domain with $C^3$ boundary. For any $\delta>0$ there exists $C_\delta\geq0$ such that
\begin{equation}\label{eqnGradPsBound}
\norm{\grad p_s(u)}_{L^2}^2 \leq
\left(\frac12+\delta\right)
\norm{\lap u}_{L^2}^2 + C_\delta \norm{\grad u}_{L^2}^2
\end{equation}
for all $u\in H^2\cap H^1_0(\Omega)$.
\end{theorem}
We refer the reader to~\cite{bblLiuLiuPego} for the proof of
Theorem~\ref{thmCommutatorEstimate}.

\subsection{\texorpdfstring{$H^1$}-equivalent coercivity on \texorpdfstring{$D(A^2)$}{D(A2)}.}
The idea behind the proof of Proposition~\ref{ppnAdjustedIP} is to use Theorem~\ref{thmCommutatorEstimate} and prove coercivity assuming the `extra' boundary condition $A u \in H^1_0$. We will later use an approximation argument to prove the Proposition for all $H^2 \cap H^1_0$ functions.

\begin{lemma}\label{lmaAdjustedIP}
For any $\epsilon > 0$ sufficiently small, there exists a constant $c = c(\Omega)$, independent of $\epsilon$, and a constant $C_\epsilon = C_\epsilon(\Omega) > 0$, depending on $\epsilon$ and $\Omega$, such that~\eqref{eqnCoercivicity} holds for all $u \in H^2 \cap H^1_0$, such that $Au \in H^1_0$.
\end{lemma}
\begin{proof}
Observe first that there exists a constant $C = C(\Omega)$ such that for all $u \in H^2 \cap H^1_0$ we have
\begin{equation}\label{eqnGradPsStupidUB}
\norm{\grad p_s(u)}_{L^2} \leq C \norm{\lap u}_{L^2}.
\end{equation}
While this immediately follows from Theorem~\ref{thmCommutatorEstimate} and the Poincar\'e inequality, we can see it directly from~\eqref{eqnPsDef} because,
due to elliptic regularity,
\begin{equation}\label{i:pslap}
\norm{\grad p_s(u)}_{L^2}^2 \leq C \left(\norm{\lap u}_{L^2}^2 + \norm{\grad \divergence u}_{L^2}^2\right) \leq C \norm{\lap u}_{L^2}^2.
\end{equation}

Now let $u \in H^2 \cap H^1_0$ be such that $Au \in H^1_0$, and $q = Q(u)$ be the unique mean zero function such that $\grad q = (I - P)u$. Then
\begin{multline}\label{eqnIp1}
\ip{u}{Au}
    = \ip{u}{-\lap u} + \ip{u}{\grad p_s}
    = \norm{\grad u}_{L^2}^2 + \ip{\grad q}{\grad p_s}\\
    \geq \norm{\grad u}_{L^2}^2 - \norm{\grad q}_{L^2} \norm{\grad p_s}_{L^2}
    \geq \norm{\grad u}_{L^2}^2 - C \norm{\grad q}_{L^2} \norm{\lap u}_{L^2}\\
    \geq \norm{\grad u}_{L^2}^2 - \frac{\epsilon}{16} \norm{\lap u}_{L^2}^2 - C_\epsilon \norm{\grad q}_{L^2}^2
\end{multline}
where the second last inequality followed from~\eqref{eqnGradPsStupidUB}, and $C_\epsilon$ is some constant depending only on $\Omega$ and $\epsilon$.

Since $Au = 0$ on $\del \Omega$ by assumption, we can integrate the $H^1$-term by parts. This gives
\begin{equation}\label{eqnIp2}
\ip{\grad u}{\grad Au} = -\ip{\lap u}{-\lap u + \grad p_s} \geq \frac{1}{8} \norm{\lap u}_{L^2}^2 - C_1 \norm{\grad u}^2
\end{equation}
where $C_1$ is the constant that arises from
Theorem~\ref{thmCommutatorEstimate}. 

Thus if $\epsilon < \frac{1}{2C_1}$, equations~\eqref{eqnIp1} and~\eqref{eqnIp2} give
\begin{equation}\label{eqnIp3}
\ip{u}{Au} + \epsilon \ip{\grad u}{\grad Au} \geq \frac{1}{2} \norm{\grad u}_{L^2}^2 + \frac{\epsilon}{16}\norm{\lap u}_{L^2}^2 - C_\epsilon \norm{\grad q}_{L^2}^2
\end{equation}

Now let $r$ be the unique mean zero function such that $\grad r = (I - P) Au$. Observe that
$$
(I - P) A u = (I - P)(-P\lap u - \grad \divergence u) = -\grad \divergence u.
$$
Since $\int_\Omega \divergence u =0$, we must have $r = -\divergence u = -\lap q$. Thus
$$
\ip{q}{r} = \ip{q}{-\lap q} = \norm{\grad q}_{L^2}^2.
$$
Combining this with~\eqref{eqnIp3} we get~\eqref{eqnCoercivicity} as desired.
\end{proof}

\subsection{Properties of the extended Stokes operator.}
Consider the extended Stokes operator $A$ as an operator from
$L^2(\Omega)$ into $L^2(\Omega)$ with domain $D(A)=H^2\cap H^1_0$.  In
this context, we recall that Proposition~\ref{ppnAdjustedIP}
asserts~\eqref{eqnCoercivicity} for all $u \in D(A)$; however,
Lemma~\ref{lmaAdjustedIP} only proves~\eqref{eqnCoercivicity} for all $u
\in D(A^2)$. To address this gap, and finish the proof of
Proposition~\ref{ppnAdjustedIP}, we need a few basic properties of the extended Stokes operator.

\begin{lemma}[Regularity and invertibility]\label{lmaAinv}
The extended Stokes operator $A$ has a compact inverse. Furthermore, there exists a constant $c = c(\Omega) > 0$ such that
\begin{equation}\label{eqnAinvBound}
\frac{1}{c} \norm{u}_{H^2} \leq \norm{Au}_{L^2} \leq c \norm{u}_{H^2},
\quad\text{for all } u \in H^2 \cap H^1_0
\end{equation}
\end{lemma}
\begin{proof}
 Our first step is to obtain estimates for the operator $A + \lambda I$ with $\lambda$ large enough. For an arbitrary $u \in H^2 \cap H^1_0$, let $f = (A + \lambda I) u$. Multiplying by $-\lap u$ and integrating gives
\begin{align*}
\lambda \norm{\grad u}_{L^2}^2 + \norm{\lap u}_{L^2}^2
    &= \int_\Omega \grad p_s \cdot \lap u \, dx - \int_\Omega f \lap u \, dx\\
    &\leq \frac{1}{2} \norm{\lap u}_{L^2} ^2 + \frac{1}{2} \norm{\grad p_s}_{L^2}^2 + \frac{1}{16} \norm{\lap u}_{L^2}^2 + 4 \norm{f}_{L^2}^2\\
    &\leq \left( \frac{1}{2} + \frac{3}{8} + \frac{1}{16} \right) \norm{\lap u}_{L^2}^2 + c \norm{\grad u}_{L^2}^2 + 4 \norm{f}_{L^2}^2
\end{align*}
where the last inequality followed from Theorem~\ref{thmCommutatorEstimate}, and $c = c(\Omega)$ is a constant. This gives
$$
(\lambda - c) \norm{\grad u}_{L^2}^2 + \frac{1}{16} \norm{\lap u}_{L^2}^2 \leq 4 \norm{f}_{L^2}^2.
$$
Thus, when  $\lambda > c$, we immediately see
\begin{equation}\label{eqnAplusLambdaIinvBound}
\norm{u}_{H^2} \leq C \norm{f}_{L^2}  = C \norm*{ \left( A + \lambda I \right) u }_{L^2}, \quad\text{when } u \in H^2 \cap H^1_0.
\end{equation}

One can use the last relation to check that $A$ is closed.  We claim
further that $A+\lambda I$ is surjective for some large enough $\lambda$. 
This can be proved by a Neumann-series perturbation argument based on 
the identity
\begin{equation}
A+\lambda I = (I+B)(\lambda I-\lap), 
\end{equation}
where $B= \grad p_s\circ (\lambda I-\lap)\inv$. That is, 
\[
Bu=\grad p_s(v), \quad v=(\lambda I-\lap)\inv u.
\]
It suffices to prove that the operator norm of $B$ on $L^2$ is strictly
less than one, if $\lambda$ is positive and large enough.
By easy energy estimates, we have that
$\lambda\norm{v}_{L^2} \le \norm{u}_{L^2}$
and $\norm{\lap v}_{L^2} \le \norm{u}_{L^2}$.
Then due to Theorem~\ref{thmCommutatorEstimate} and interpolation, we have
\[
\norm{Bu}_{L^2}^2
= \int_\Omega|\grad p_s(v)|^2
\leq 
\beta \norm{\lap v}_{L^2}^2 
+C_\beta \norm{v}_{L^2}^2 
\le 
\left(\beta + C_\beta\lambda^{-2}\right) \norm{u}_{L^2}^2,
\]
and the coefficient on the right is less than 1 for $\lambda$ large
enough. Thus $I+B$ is an isomorphism on $L^2$, hence $A+\lambda I$ is
surjective.  

Further, the Rellich-Kondrachov compact embedding
theorem and the bound \eqref{eqnAplusLambdaIinvBound} imply that  
$A + \lambda I$ has compact inverse.
Since we have shown that the resolvent of $A$ contains at least one element with a compact inverse, the spectrum of $A$ consists only of (isolated) eigenvalues, of finite multiplicity (see for instance~\cite{bblKato}*{Theorem III.6.29}). Thus to prove invertibility of $A$, it  suffices to show that $0$ is not an eigenvalue of $A$.

To see this, 
suppose $u \in D(A)$ is such that $A u = 0$. Then $-P \lap u = \grad \divergence u \in L^2(\Omega)$. Since the range of the Leray projection (by definition) is orthogonal to gradients, we must have $\grad \divergence u = P \lap u = 0$, and hence $\divergence u$ must be constant. Since $\int_\Omega \divergence u = \int_{\del \Omega} u \cdot \nu = 0$, this forces $\divergence u =0$. Thus $u = Pu$, and is orthogonal to gradients. Since $A u = 0$, we have
$$
0 = \int_\Omega u \cdot A u \, dx = - \int_\Omega u \cdot \lap u \, dx + \int_\Omega P u \cdot \grad p_s(u) \, dx = \norm{\grad u}_{L^2}^2 + 0,
$$
forcing $u = 0$. Hence $0$ is not an eigenvalue of $A$, and we conclude that $A$ is invertible.

It remains to establish~\eqref{eqnAinvBound}. The upper bound follows immediately from~\eqref{eqnAdef} and~\eqref{eqnGradPsStupidUB}. To prove the lower bound, observe first that boundedness of $A\inv$ implies
\begin{equation}\label{eqnAinvBoundL2}
\norm{A\inv u}_{L^2} \leq C \norm{u}_{L^2},
\end{equation}
for some constant $C = C(\Omega)$, which we subsequently allow to change from line to line. Thus using the operator identity
$$
A\inv = (A + \lambda I)\inv \left( I + \lambda A\inv \right),
$$
and the inequalities~\eqref{eqnAplusLambdaIinvBound}, \eqref{eqnAinvBoundL2} we see
\begin{multline*}
\norm{A\inv u}_{H^2} =  \norm{ (A + \lambda I)\inv \left( I + \lambda A\inv \right) u}_{H^2} \\
\leq C \norm{\left( I + \lambda A\inv \right) u}_{L^2} \leq C (1 + C \lambda) \norm{u}_{L^2}
\end{multline*}
proving the lower bound in~\eqref{eqnAinvBound}.
\end{proof}

\begin{lemma}\label{lmaDensityOfDA2}
For the extended Stokes operator, $D(A^2)$ is dense in $D(A)$.
\end{lemma}
\begin{proof}
Let $u \in D(A)$, and $v = Au$. Since $v \in L^2(\Omega)$, we can find $v_n \in H^2 \cap H^1_0$ such that $(v_n) \to v$ in $L^2$. Since $D(A) = H^2 \cap H^1_0$ by letting $u_n \defeq A\inv v_n$, we have  $u_n \in D(A^2)$. Finally, by Lemma~\ref{lmaAinv} we see
$$
\norm{u_n - u}_{H^2} \leq c \norm{A u_n - A u}_{L^2} = c \norm{v_n - v}_{L^2} \to 0,
$$
concluding the proof.
\end{proof}

\subsection{\texorpdfstring{$H^1$}{H1}-equivalent coercivity on \texorpdfstring{$D(A)$}{D(A)}.}
Lemmas~\ref{lmaAdjustedIP} and~\ref{lmaDensityOfDA2} quickly imply Proposition~\ref{ppnAdjustedIP}.
\begin{proof}[Proof of Proposition~\ref{ppnAdjustedIP}]
Let $u \in D(A)$. By Lemma~\ref{lmaDensityOfDA2}, there exists a sequence $u_n \in D(A^2)$ such that $(u_n) \to u$ in $H^2$. By Lemma~\ref{lmaAdjustedIP}, there exists constants $c(\Omega), C_\epsilon(\Omega) > 0$ such that
$$
\ip{u_n}{Au_n}_{\epsilon} \geq \frac{1}{c} \left( \norm{\grad u_n}_{L^2}^2 + \epsilon \norm{\lap u_n}_{L^2}^2 + C_\epsilon \norm{\grad q_n}_{L^2}^2\right),
$$
where $q_n$ is the unique, mean-zero function such that $\grad q_n = (I - P) u_n$. Since $(u_n) \to u$ in $H^2$, taking limits as $n \to \infty$ yields~\eqref{eqnCoercivicity}. Now using the Poincar\'e inequality,~\eqref{eqnCoercivicity2} follows.
\end{proof}

\subsection{\texorpdfstring{$\hdiv$}{H-div} equivalent coercivity.}
We conclude this section by proving coercivity under an
$\hdiv$-equivalent inner product. 
The rest of this paper is independent of this result and its proof.
\begin{proposition}[$\hdiv$-coercivity]\label{ppnAdjustedIPPrime}
Let $\Omega \subset \R^3$ be a $C^3$ domain. There exists positive constants $\epsilon_0 > 0$ and $c = c(\Omega)$ such that for all $\epsilon \in (0, \epsilon_0)$, there exists a constant $C_{\epsilon} = C_\epsilon(\Omega) > 0$ such that the following hold.
\begin{enumerate}
\item Let $\ip{\cdot}{\cdot}_{\epsilon}'$ be defined by
\begin{multline}\label{eqnAdjustedIPPrime}
\ip{u}{v}_{\epsilon}' \defeq \ip{u}{v} + \epsilon \ip{\divergence u}{\divergence v} + C_{\epsilon} \ip{A\inv u}{A\inv v}_\epsilon \mathop-\\
    \mathop- \ip{u}{\grad p_s( A\inv v)} - \ip{\grad p_s(A\inv u)}{v},
\end{multline}
where $\ip{u}{v}$ is the standard inner product on $L^2(\Omega)$, and $\ip{\cdot}{\cdot}_\epsilon$ denotes the inner product from Proposition~\ref{ppnAdjustedIP}. Then
\begin{equation}\label{eqnPositivity}
\frac{1}{c} \norm{u}_{L^2} + \epsilon \norm{\divergence u}_{L^2}^2 \leq \ip{u}{u}'_{\epsilon} \leq c (1 + C_{\epsilon}) \norm{u}_{L^2}^2 + \epsilon \norm{\divergence u}_{L^2}^2
\end{equation}
for any $u \in L^2(\Omega)$ with $\divergence u \in L^2(\Omega)$ and
$u\cdot\nu=0$ on $\partial\Omega$.
\item For any $u \in H^2\cap H^1_0$ with $\divergence u \in H^1$, we have
\begin{gather}
\label{eqnPositivityOfA} \ip{u}{A u}'_{\epsilon} \geq \norm{\grad u}_{L^2}^2 + \frac{\epsilon}{2} \norm{\grad \divergence u}_{L^2}^2\\
\llap{\text{and}\qquad}
\label{eqnPositivityOfA2} \ip{u}{A u}'_{\epsilon} \geq \frac{1}{c} \ip{u}{u}'_{\epsilon}.
\end{gather}
\end{enumerate}
\end{proposition}
%
\begin{proof} 
We begin by proving~\eqref{eqnPositivity} for all $u \in D(A)$. Density of $D(A)$ in $L^2$ and a standard approximation argument will now establish~\eqref{eqnPositivity} for all $u \in \hdiv$. We will assume $C_{\epsilon}$ and $c$ are constants that can change from line to line, provided their dependence on parameters is as required in the Proposition. Let $u \in D(A)$, and $v = A\inv u$. Then from~\eqref{eqnAdjustedIPPrime} we have
\begin{equation}\label{eqnIpUUEpsilonPrime}
\begin{aligned}
\ip{u}{u}'_{\epsilon} &= \ip{Av}{Av} - 2 \ip{Av}{\grad p_s(v)} + \epsilon \norm{\divergence u}_{L^2}^2 + C_{\epsilon} \ip{v}{v}_\epsilon \\
    &= \ip{-\lap v + \grad p_s(v)}{-\lap v - \grad p_s(v)} + \epsilon \norm{\divergence u}_{L^2}^2 + C_{\epsilon} \ip{v}{v}_\epsilon\\
    &= \norm{\lap v}_{L^2}^2 - \norm{\grad p_s(v)}_{L^2}^2 + \epsilon \norm{\divergence u}_{L^2}^2 + C_{\epsilon} \ip{v}{v}_\epsilon\\
    &\geq \frac{1}{4} \norm{\lap v}_{L^2}^2 - c \norm{v}_{L^2}^2 + \epsilon \norm{\divergence u}_{L^2}^2 + C_{\epsilon} \ip{v}{v}_\epsilon
\end{aligned}
\end{equation}
where the last inequality followed from Theorem~\ref{thmCommutatorEstimate} and interpolation. Now since $Av = u$, we immediately see
$$
\norm{u}_{L^2}^2 = \norm{- P \lap v - \grad \divergence v}_{L^2}^2 \leq c \norm{\lap v}_{L^2}^2.
$$
Finally, by definition of $\ip{v}{v}_\epsilon$, we have $\ip{v}{v}_\epsilon \geq \norm{v}_{L^2}^2$. Thus if $C_{\epsilon} > c$,  the lower bound in equation~\eqref{eqnPositivity} will hold for all $\epsilon > 0$.

For the upper bound in~\eqref{eqnPositivity}, observe that by definition of $\grad p_s$, and Lemma~\ref{lmaAinv} we have
$$
\norm{\grad p_s A\inv u}_{L^2} \leq c \norm{A\inv u}_{H^2} \leq c \norm{u}_{L^2}.
$$
Combined with the estimate $\norm{A\inv u}_{L^2} \leq c \norm{u}_{L^2}$, which is also a consequence of Lemma~\ref{lmaAinv}, we immediately obtain the upper bound in~\eqref{eqnPositivity}.\smallskip

Finally, it remains to prove the inequality~\eqref{eqnPositivityOfA}. We will prove~\eqref{eqnPositivityOfA} for $u \in D(A^2)$; since $D(A^2)$ is dense in $D(A)$, the same approximation argument from the proof of 
Proposition~\ref{ppnAdjustedIP}
will show that~\eqref{eqnPositivityOfA} holds on $D(A)$. In keeping with the above notation, we again set $v = A \inv u$. This gives
$$
\ip{u}{Au}'_{\epsilon} = \ip{u}{Au} - \ip{u}{\grad p_s(u)} - \ip{\grad p_s(v)}{Au} + \epsilon \ip{\divergence u}{\divergence Au} + C_{\epsilon} \ip{v}{Av}_\epsilon.
$$
We deal with the terms on the right individually. Combining the first two terms, the dangerous term involving the Stokes pressure cancels. This gives
$$
\ip{u}{Au} - \ip{u}{\grad p_s(u)} = \ip{u}{-\lap u} = \norm{\grad u}_{L^2}^2.
$$
For the third term,
$$
-\ip{\grad p_s(v)}{Au}  = -\ip{\grad p_s(v)}{(I - P) Au} =  \ip{\grad p_s(v)}{\grad \divergence u}.
$$
For the fourth term, observe that if $u \in D(A^2)$, then $Au = 0$ on $\del \Omega$. Thus integrating by parts gives
$$
\ip{\divergence u}{\divergence Au} = - \ip{\grad \divergence u}{A u} =  \norm{\grad \divergence u}_{L^2}^2.
$$
Combining these identities we have
\begin{align*}
\ip{u}{Au}'_{\epsilon} &= \norm{\grad u}_{L^2}^2 + \ip{\grad p_s(v)}{\grad \divergence u} + \epsilon \norm{\grad \divergence u}_{L^2}^2 + C_{\epsilon} \ip{v}{Av}_\epsilon\\
     & \geq \norm{\grad u}_{L^2}^2 + \frac{\epsilon}{2} \norm{\grad \divergence u}_{L^2}^2 - \frac{1}{2 \epsilon} \norm{\grad p_s(v)}_{L^2}^2 + C_{\epsilon} \ip{v}{Av}_\epsilon\\
     & \geq \norm{\grad u}_{L^2}^2 + \frac{\epsilon}{2} \norm{\grad \divergence u}_{L^2}^2 - \frac{c}{\epsilon} \norm{\lap v}_{L^2}^2 + C_{\epsilon} \ip{v}{Av}_\epsilon\\
& \geq  \norm{\grad u}_{L^2}^2 + \frac{\epsilon}{2} \norm{\grad \divergence u}_{L^2}^2 + \left( \frac{\epsilon C_{\epsilon}}{2 c_1} - \frac{c}{\epsilon} \right) \norm{\lap v}_{L^2}^2 + \frac{C_\epsilon}{2} \ip{v}{Av}_\epsilon
\end{align*}
where the last inequality followed from Proposition~\ref{ppnAdjustedIP}, and $c_1$ is the constant in~\eqref{eqnCoercivicity}. We note that Proposition~\ref{ppnAdjustedIP} guarantees that $c_1$ is independent of $\epsilon$. Now we choose $C_{\epsilon}$ large enough so that $\frac{\epsilon C_{\epsilon}}{2c_1} - \frac{c}{\epsilon} > 1$, giving
\begin{equation}\label{eqnIpUAuEpsilonPrime}
\ip{u}{Au}'_{\epsilon} \geq
\norm{\grad u}_{L^2}^2 + \frac{\epsilon}{2} \norm{\grad \divergence u}_{L^2}^2 + \norm{\lap v}_{L^2}^2 + \frac{C_\epsilon}{2} \ip{v}{Av}_\epsilon,
\end{equation}
from which inequality \eqref{eqnPositivityOfA} follows.

Finally for~\eqref{eqnPositivityOfA2}, observe that from~\eqref{eqnIpUUEpsilonPrime} we have
\begin{align*}
\ip{u}{u}'_{\epsilon} &= \norm{\lap v}_{L^2}^2 - \norm{\grad p_s(v)}_{L^2}^2 + \epsilon \norm{\divergence u}_{L^2}^2 + C_{\epsilon} \ip{v}{v}_\epsilon\\
    &\leq c_2 \norm{\lap v}_{L^2}^2 + \epsilon \norm{\divergence u}_{L^2}^2 + C_{\epsilon} \ip{v}{v}_\epsilon,
\end{align*}
for some constant $c_2 = c_2(\Omega)$. Now using Proposition~\ref{ppnAdjustedIP} and~\eqref{eqnIpUAuEpsilonPrime}, the inequality~\eqref{eqnPositivityOfA2} follows.
\end{proof}

\section{Decay of a non-quadratic form energy.}\label{sxnExtStokesExpDecay}
This section comprises the proof of Proposition~\ref{ppnStokesEnergyDecay2}, addressing the long time behaviour of solutions to the extended Stokes equations~\eqref{eqnExtendedStokes}. The result and proof are independent of the rest of this paper.
\begin{proof}[Proof of Proposition~\ref{ppnStokesEnergyDecay2}]
In this proof, we use $C$ to denote an intermediate constant that depends only on $\Omega$ whose value can change from line to line. We use $C_1, C_2, \dots$ to denote fixed positive constants that depend only on $\Omega$, whose values \emph{do not} change from line to line.

As usual, let $q = Q(u)$ be the unique mean zero function such that $\grad q = (I - P) u$. We begin by establishing the energy inequalities
\begin{gather}
\label{eqnUEnergy}
    \del_t \norm{u}_{L^2}^2 + 2 \norm{\grad u}_{L^2}^2 \leq C_1 \norm{\grad q}_{L^2} \norm{\lap u}_{L^2},\\
\label{eqnGradUEnergy2}
    \del_t \norm{\grad u}_{L^2}^2 + \frac{1}{4} \norm{\lap u}_{L^2}^2 \leq C \norm{\grad u}_{L^2}^2,\\
\label{eqnDivEnergy1}
    \del_t \norm{\grad q}_{L^2}^2 +  2 \norm{\divergence u}_{L^2}^2 \leq 0.
\end{gather}
Before proving the above inequalities, we remark that the $L^2$ balance~\eqref{eqnUEnergy} does not close by itself. On the other hand, the $H^1$ balance~\eqref{eqnGradUEnergy2} closes, but does not give decay. A combination of the norms, however, gives us the desired exponential decay.\smallskip

For the proof of~\eqref{eqnUEnergy}, multiply~\eqref{eqnExtendedStokes} by $u$ and integrate over $\Omega$ to obtain
\begin{multline*}
\frac{1}{2} \del_t \norm{u}_{L^2}^2 + \norm{\grad u}_{L^2}^2
    = - \int_{\Omega} u \cdot \grad p_s
    = - \int_{\Omega} \grad q \cdot \grad p_s
    \leq C_1 \norm{\grad q}_{L^2} \norm{\lap u}_{L^2},
\end{multline*}
where we used \eqref{i:pslap}.
This establishes~\eqref{eqnUEnergy}.\smallskip

Turning to~\eqref{eqnGradUEnergy2}, we multiply~\eqref{eqnExtendedStokes} by $-\lap u$ and integrate over $\Omega$ to obtain
\begin{equation}\label{eqnGradUEnergy1}
\frac{1}{2} \del_t \norm{\grad u}_{L^2}^2 + \norm{\lap u}_{L^2}^2 = \int_\Omega \grad p_s \cdot \lap u.
\end{equation}
Using Theorem~\ref{thmCommutatorEstimate}, we know that for any $\delta > 0$ there exists a constant $C_\delta = C_\delta(\Omega)$ such that
$$
\norm{\grad p_s}_{L^2}^2 \leq \left( \frac{1 + \delta }{2} \right) \norm{\lap u}_{L^2}^2 + C_\delta \norm{\grad u}_{L^2}^2.
$$
Hence
\begin{equation}\label{eqnGradPsLapUBound}
\abs*{\int_\Omega \grad p_s \cdot \lap u}
    \leq \frac{1}{2} \norm{\grad p_s}_{L^2}^2 + \frac{1}{2}\norm{\lap u}_{L^2}^2
    \leq \left( \frac{3 + \delta}{4} \right) \norm{\lap u}_{L^2}^2+  \frac{1}{2} C_\delta \norm{\grad u}_{L^2}^2.
\end{equation}
Choosing $\delta = \frac{1}{2}$, equation~\eqref{eqnGradUEnergy1} reduces to~\eqref{eqnGradUEnergy2} as desired.\smallskip

Finally for~\eqref{eqnDivEnergy1}, we apply $(I - P)$ to~\eqref{eqnExtendedStokes} to get
$$
\del_t \grad q + (I - P)(- \lap u + \grad p_s(u) ) = 0.
$$
Since $\divergence u = \lap q$ and
$$
(I - P) (-\lap u + \grad p_s(u) ) = (I - P)( P \lap u - \grad \divergence u) = \grad \divergence u = \grad \lap q,
$$
we see
$$
\del_t \grad q - \lap \grad q = 0,
$$
and hence $\del_t q - \lap q = C(t)$, where $C$ is constant in space. Now, since $u = 0$ on $\del \Omega$, we must have $\frac{\del q}{\del \nu} = 0$ on $\del \Omega$. This means $\int_\Omega \lap q = 0$; since $\int_\Omega q = 0$ by our choice of $q$, we must have $C(t) = 0$. Thus we obtain
\begin{equation}\label{eqnHeatQ}
\del_t q - \lap q = 0 \quad\text{in }\Omega,\\
\qquad\text{with }
\frac{\del}{\del \nu} q = 0 \text{ for }x \in \del \Omega.
\end{equation}
Multiplying by $-\lap q$ and integrating over $\Omega$ gives~\eqref{eqnDivEnergy1} as desired.\medskip

Now we combine~\eqref{eqnUEnergy}--\eqref{eqnDivEnergy1} to obtain the desired exponential decay. First from~\eqref{eqnGradUEnergy2}, \eqref{eqnDivEnergy1} and the Poincar\'e inequality we have
\begin{gather*}
\del_t \norm{\grad q}_{L^2} + \frac{1}{C_2} \norm{\grad q}_{L^2} \leq 0,\\
\del_t \norm{\grad u}_{L^2} + \frac{1}{C_2} \norm{\lap u}_{L^2} \leq C_3\norm{\grad u}.
\end{gather*}
Thus
$$
\del_t \left( \norm{\grad u}_{L^2} \norm{\grad q}_{L^2} \right) + \frac{1}{C_2} \norm{\lap u}_{L^2} \norm{\grad q}_{L^2} \leq C_4 \norm{\grad u}_{L^2} \norm{\grad q}_{L^2}.
$$
where $C_4 = C_3 - \frac{1}{C_2}$. Using this in~\eqref{eqnUEnergy} we see
\begin{multline*}
\del_t \left( \norm{u}_{L^2}^2 + C_2 (C_1 + 1) \norm{\grad u}_{L^2} \norm{\grad q}_{L^2} \right) \mathop + \\
    + 2 \norm{\grad u}_{L^2}^2 + \norm{\lap u}_{L^2} \norm{\grad q}_{L^2} \leq C_5 \norm{\grad u}_{L^2} \norm{\grad q}_{L^2}
\end{multline*}
where $C_5 = C_2 (C_1 + 1) C_4$. Now letting $c_1 = C_2 ( C_1 + 1)$, and $c_2 = c_2(\Omega, \epsilon)$ to be chosen later, we see that
\begin{multline*}
\del_t \left( \norm{u}_{L^2}^2 + c_1 \norm{\grad u}_{L^2} \norm{\grad q}_{L^2} + c_2 \norm{\grad q}_{L^2}^2 \right) \mathop+\\
    + 2 \norm{\grad u}_{L^2}^2 + \norm{\lap u}_{L^2} \norm{\grad q}_{L^2} + 2 c_2 \norm{\divergence u}_{L^2}^2
    \leq \epsilon \norm{\grad u}_{L^2}^2 + \frac{C_5^2}{2 \lambda_1 \epsilon} \norm{\divergence u}_{L^2}^2
\end{multline*}
where $\lambda_1$ is the best constant in the Poincar\'e inequality
$$
\lambda_1 \norm{\grad q}_{L^2}^2 \leq \norm{\lap q}_{L^2}^2 = \norm{\divergence u}_{L^2}^2.
$$
Thus choosing $c_2 = \frac{C_5^2}{4 \lambda_1 \epsilon} + \frac{1}{2}$, we obtain~\eqref{eqnNonLinearEnergyDecay}.
\end{proof}

\section{Global stability of time discretization for the extended Stokes equations}\label{sxnTimeDiscrete}

We devote this section to proving Proposition~\ref{prop:disc}. The main idea again is similar: to introduce a stabilizing, high order term in the definition of the energies.
\begin{proof}[Proof of Proposition~\ref{prop:disc}]
In the following we use $C_1, C_2, \dots$ to denote fixed positive constants that depend only on $\Omega$, whose values \emph{do not} change from line to line and a generic constant $C$ whose value might change from one line to the next, depending only on $\Omega$.

Let $p_s^n$ be the Stokes pressure for $u^n$ hence  $\grad p_s^n=(\lap P-P\lap)u^n$ and  thus
$$
\int_\Omega \grad p^n_s\cdot\grad\varphi=\int_\Omega (\lap u^n-\grad\divergence u^n)\cdot\grad\varphi,\quad\forall\varphi\in H^1(\Omega).
$$
Using~\eqref{disceq:weakpressure} with $\varphi = p^n$ and combining it with the last relation we obtain
\begin{equation}\label{disceq:pressures}
\norm{\grad p^n}_{L^2}\leq \norm{\grad p^n_s}_{L^2}+\norm{f^n}_{L^2}.
\end{equation}

We derive first the discrete $H^1$ estimate just as in \cite{bblLiuLiuPego}. 
Taking the $L^2$ inner product of~\eqref{disceq:weakStokes} with $-\lap u^{n+1}$ gives
\begin{align*}
\MoveEqLeft \frac{1}{2\delta t}\left(\norm{\grad u^{n+1}}_{L^2}^2-\norm{\grad u^n}_{L^2}^2+\norm{\grad u^{n+1}-\grad u^n}_{L^2}^2\right)+\norm{\lap u^{n+1}}_{L^2}^2\\
&\leq \norm{\lap u^{n+1}}_{L^2}\left(2\norm{f^n}_{L^2}+\norm{\grad p_s^n}_{L^2}\right)\\
&\leq \frac{\epsilon_1}{2}\norm{\lap u^{n+1}}_{L^2}^2+
\frac{2}{\epsilon_1}\norm{f^n}_{L^2}^2+\frac{1}{2}(\norm{\lap u^{n+1}}_{L^2}^2+\norm{\grad p_s^n}_{L^2}^2)
\end{align*}
for all $\epsilon_1 > 0$. This implies
\begin{multline}\label{disceq:H1est}
\frac{1}{\delta t}\left(\norm{\grad u^{n+1}}_{L^2}^2-\norm{\grad u^n}_{L^2}^2\right)+\frac{1}{\delta t}\norm{\grad u^{n+1}-\grad u^n}_{L^2}^2 \mathop+\\
    + (1-\epsilon_1)\norm{\lap u^{n+1}}_{L^2}^2
    \leq \frac{4}{\epsilon_1}\norm{f^n}_{L^2}^2+\norm{\grad p^n_s}_{L^2}^2.
\end{multline}

Fix any $\beta\in (\frac{1}{2},\frac{2}{3})$. By Theorem~\ref{thmCommutatorEstimate} we have
$$
\norm{\grad p^n_s}_{L^2}^2\leq \frac{3}{2}\beta\norm{\lap u^n}_{L^2}^2+\frac{2C_\beta}{3}\norm{\grad u^n}_{L^2}^2.
$$
Using this in~\eqref{disceq:H1est} and dividing by $2C_\beta$ we get
\begin{multline}\label{disceq:H1est+}
\frac{1}{2C_\beta\delta t}\left(\norm{\grad u^{n+1}}_{L^2}^2-\norm{\grad u^n}_{L^2}^2+\norm{\grad u^{n+1}-\grad u^n}_{L^2}^2\right)\\
+\frac{(1-\epsilon_1)}{2C_\beta}\left(\norm{\lap u^{n+1}}_{L^2}^2-\norm{\lap u^n}_{L^2}^2\right)+\frac{(2-2\epsilon_1-3\beta)}{4C_\beta}\norm{\lap u^n}_{L^2}^2\\
\leq \frac{2}{\epsilon_1C_\beta}\norm{f^n}_{L^2}^2+\frac{1}{3}\norm{\grad u^n}_{L^2}^2,
\end{multline}
and we may assume that $\epsilon_1>0$ is small enough so that $1-\epsilon_1-\frac{3}{2}\beta>0$.

We continue by obtaining the discrete $L^2$ estimate. We dot the equation \eqref{disceq:weakStokes} by $u^{n+1}$  in $L^2$ and obtain
\begin{align}
\MoveEqLeft
\begin{multlined}
\nonumber
\frac{1}{2\delta t}\left(\norm{u^{n+1}}_{L^2}^2-\norm{u^n}_{L^2}^2+\norm{u^{n+1}-u^n}_{L^2}^2\right)\\
+\frac{1}{3}\norm{\grad u^{n+1}}_{L^2}^2+\frac{2}{3}\left(\norm{\grad u^{n+1}}_{L^2}^2-\norm{\grad u^n}_{L^2}^2\right)+\frac{2}{3}\norm{\grad u^n}_{L^2}^2
\end{multlined}\\
    \nonumber
    &= \int_\Omega (f^n-\grad p^n) \cdot u^{n+1}
= \int_\Omega (Pf^n)\cdot u^{n+1}-\int_\Omega \grad p_s(u^n)\cdot\grad q^{n+1}
\\ \nonumber
&\leq 
\norm{f^n}_{L^2} 
\norm{u^{n+1}}_{L^2}
+
\norm{\grad p_s(u^n)}_{L^2}
\norm{\grad q^{n+1}}_{L^2}
\\
    &\leq 
\frac{\lambda_0}{3}\norm{u^{n+1}}_{L^2}^2
+ \frac{2-2\epsilon_1-3\beta}{8C_\beta} \norm{\lap u^n}_{L^2}^2
+
C_1( \norm{f^n}_{L^2}^2 +\norm{\grad q^{n+1}}_{L^2}^2),
    \label{disceq:L2est}
\end{align}
where $q^{n+1} = Q(u^{n+1})$, and $\lambda_0$ is the principal eigenvalue of the Laplacian on $\Omega$ with zero Dirichlet boundary conditions.

Since $p^n=p_s^n+Q(f^n)$, by applying $I-P$ to
\eqref{disceq:weakStokes} we find that 
$q^n$ satisfies the time-discrete inhomogeneous heat equation
\begin{equation}
\frac1{\delta t}\left( q^{n+1}-q^n\right) -\lap q^{n+1} = 
p_s^{n+1}-p_s^n .
\end{equation}
Then we find after testing with $-\lap q^{n+1}$ that, 
as above (and as in~\cite{bblLiuLiuPego}*{page 1477}), 
\begin{equation}\label{eqn12191}
\frac{1}{\delta t}\left(\norm{\grad q^{n+1}}_{L^2}^2-\norm{\grad q^n}_{L^2}^2\right)+\norm{\lap q^{n+1}}^2\leq \norm{p_s^{n+1}-p_s^n}_{L^2}^2.
\end{equation}
Also, since $p^n$ satisfies a Neumann boundary value problem, 
we have the estimate 
\begin{equation}\label{eqn12192}
\norm{p_s^{n+1}-p^n_s}_{L^2}^2\leq C\norm{u^{n+1}-u^n}_{L^2}^{\frac{1}{2}}\norm{u^{n+1}-u^n}_{H^2}^{\frac{3}{2}}.
\end{equation}
Now choose $C_2$ large enough to ensure
$$
C_1\norm{\grad q^{n+1}}_{L^2}^2\leq \frac{ C_2}{2}\norm{\lap q^{n+1}}_{L^2}^2,
$$
and $\epsilon_2$ small enough so that $4\epsilon_2<1-\epsilon_1-\beta$. Combining~\eqref{eqn12191} and~\eqref{eqn12192}, we obtain
\begin{align}\label{disceq:divest}
&\frac{1}{\delta t}\left(\norm{\grad q^{n+1}}_{L^2}^2-\norm{\grad q^n}_{L^2}^2\right)+\norm{\lap q^{n+1}}^2\\
&\qquad    \leq \frac{\epsilon_2}{4C_\beta C_2}\norm{\lap u^{n+1}-\lap u^n}_{L^2}^2+C_3\norm{u^{n+1}-u^n}_{L^2}^2
\nonumber\\
 &\qquad   \leq \frac{\epsilon_2}{4C_\beta C_2}\norm{\lap u^{n+1}-\lap u^n}_{L^2}^2+C_4\norm{\grad u^{n+1}-\grad u^n}_{L^2}^2
\nonumber
\end{align}
for large enough constants $C_3$ and $C_4$.

Assume that $\delta t$ is small enough so that
$(2C_\beta C_2C_4)\delta t<1$.
Multiplying \eqref{disceq:divest}  by $C_2$, and adding it to~\eqref{disceq:H1est+} and \eqref{disceq:L2est} then gives
\begin{multline}\label{disceq:finalneq}
\frac{C_2}{\delta t}\left(\norm{\grad q^{n+1}}_{L^2}^2-\norm{\grad q^n}_{L^2}^2\right)
+\frac{C_2}{2}\norm{\lap q^{n+1}}_{L^2}^2 
+\frac{1}{2\delta t}\left(\norm{u^{n+1}}_{L^2}^2-\norm{u^n}_{L^2}^2\right)
\\
+\left(\frac{1}{2C_\beta\delta t}+\frac23\right)\left(\norm{\grad u^{n+1}}_{L^2}^2-\norm{\grad u^n}_{L^2}^2\right)
+\frac{1}{3}\norm{\grad u^n}_{L^2}^2\\
+\frac{1-\epsilon_1}{2C_\beta}\left(\norm{\lap u^{n+1}}_{L^2}^2-\norm{\lap u^n}_{L^2}^2\right)
+\frac{2-2\epsilon_1-3\beta}{8C_\beta}\norm{\lap u^n}_{L^2}^2\\
\leq \frac{\epsilon_2}{2C_\beta}\left(\norm{\lap u^{n+1}}^2+\norm{\lap u^n}_{L^2}^2\right)
+\left(\frac{2}{\epsilon_1 C_\beta}+C_1\right)\norm{f^n}_{L^2}^2.
\end{multline}
Now summing from $n=0$ to $N$ the last inequality gives (for small enough $\delta t$, and for a suitable constant $C>0$) the claimed inequality \eqref{eqnDiscEnergyIneq}.


We rearrange \eqref{disceq:finalneq} and  obtain
\begin{multline}\label{disceq:largerecurrence}
\frac{1}{\delta t}\left(\|\grad q^{n+1}\|_{L^2}^2-(1-\hat C\delta t)\|\grad q^n\|_{L^2}^2\right) \mathop+\\
    \frac{1}{\delta t}\left(\|u^{n+1}\|_{L^2}^2-(1-\hat C\delta t)\|u^n\|_{L^2}^2\right) + \frac{1}{\delta t}\left(\|\grad u^{n+1}\|_{L^2}^2-(1-\hat C\delta t)\|\grad u^n\|_{L^2}^2\right) \mathop+\\
    +\left (\|\grad u^{n+1}\|_{L^2}^2-(1-\hat C\delta t)\|\grad u^n\|_{L^2}^2\right) +\\
    + (1+\delta t)\left(\|\lap u^{n+1}\|_{L^2}^2-(1-\hat C\delta t)\|\lap u^n\|_{L^2}^2\right) \mathop+\\
    \left(\|\lap q^{n+1}\|_{L^2}^2-(1-\hat C\delta t)\|\lap q^n\|^2\right) \mathop+ (1-\hat C\delta t)\|\lap q^n\|_{L^2}^2
\leq C\|f^n\|_{L^2}^2
\end{multline}
provided that $\delta t$ is small enough, for suitable constants $C$ and $\hat C$.
Defining \begin{multline*}
a_n\defeq\frac{1}{\delta t}\|\grad q^n\|_{L^2}^2+\frac{1}{\delta t}\|u^n\|_{L^2}^2+\frac{1}{\delta t}\|\grad u^n\|_{L^2}^2\\
+\|\grad u^n\|_{L^2}^2+(1+\delta t)\|\lap u^n\|_{L^2}^2+\|\lap q^n\|_{L^2}^2,
\end{multline*}
\eqref{disceq:largerecurrence} becomes
\begin{equation}\label{disceq:smallrecurrence}
a_{n+1}-(1-\hat C\delta t)a_n\leq  C\|f_n\|_{L^2}^2 .
\end{equation}
Solving this recurrence relation yields~\eqref{eqnDiscExpDecay}.
\end{proof}
\section{Small data global existence for the extended Navier-Stokes equations}\label{sxnSmallData}
This section is devoted to the proof of a long time, small data existence result (Theorem~\ref{thmSmallDataGexist}) for the system~\eqsysENS. As bounds for the linear terms have already been established (Proposition~\ref{ppnAdjustedIP}), we begin with a bound on the nonlinear term. When obtaining energy estimates for solutions to~\eqref{eqnNS}, the explicit, exponential decay of $\divergence u$ allows sharper estimates for many terms. However, in order to exploit coercivity of the linear terms, we are forced to use an $H^1$-equivalent inner product. In this case, the `worst' term that arises from the nonlinearity isn't aided by decay of $\divergence u$, and must be estimated brutally. Consequently, estimating the remaining terms similarly doesn't weaken the final result. Thus, we begin with a lemma that provides a `brutal' estimate on the nonlinearity.

\begin{lemma}\label{lmaNonGenEst}
Let $f, g, h\in H^2\cap H^1_0(\Omega)$ with $\Omega\subset \mathbb{R}^d,d=2,3$  a bounded domain with $C^3$ boundary. Then there exists a constant $C = C(\Omega) > 0$ such that%
\footnote{Our estimates on the nonlinear term are not optimal. Using `optimal' estimates here would be at the expense of simplicity, and obfuscate the main idea. Further, the `optimal' estimates are still insufficient to prove global existence without a smallness assumption on the initial data.}
$$
\abs{\ip{P((f\cdot\grad) g)}{h}_\epsilon} = C \norm{f}_{H^1}\norm{\grad g}_{H^{\frac{1}{2}} }\norm{\lap h}_{L^2}.
$$
\end{lemma}
\begin{proof}
Observe first that
$$
\ip{P((f\cdot\grad) g)}{h}_\epsilon = \ip{P((f\cdot\grad) g)}{h} + \epsilon \ip{\grad P((f\cdot\grad) g)}{\grad h},
$$
since $(I - P) P((f\cdot\grad) g) = 0$. Thus to prove the lemma, it suffices to show the estimates
\begin{gather}
\label{abcnonL2}
\abs{\ip{P((f\cdot\grad) g)}{h}} \leq C \norm{f}_{H^1}\norm{\grad g}_{H^{\frac{1}{2}} }\norm{h}_{L^2}\\
\llap{and\qquad}
\label{abcnonH1}
    \abs{\ip{\grad P((f\cdot\grad)g)}{\grad h}}\leq C \norm{f}_{H^1}\norm{\grad g}_{H^{\frac{1}{2}} }\norm{\lap h}_{L^2}.
\end{gather}
for some constant $C = C(\Omega)$.

The inequality~\eqref{abcnonL2} follows directly from the Sobolev embedding theorem. Indeed, for any three functions $f_1, f_2, f_3$, we know%
\footnote{For non-integer values of $s$, we define the fractional Sobolev norms by interpolation. See for instance~\cite{bblConstFoias}*{Page 50}.}
\begin{equation}\label{eqnIntFGH}
\abs*{\int_\Omega f_1 f_2 f_3} \leq C \norm{f_1}_{H^{s_1}}\norm{f_2}_{H^{s_2}}\norm{f_3}_{H^{s_3}}
\end{equation}
provided $0\leq s_i \leq 3$, $s_1 + s_2 + s_3 \geq \frac{d}{2}$ and at least two of $s_1, \dots, s_3$ are non-zero (see for instance the proof of Proposition 6.1\ in~\cite{bblConstFoias}). Choosing $s_1=1$, $s_2 = 1/2$ and $s_3 = 0$, we have
$$
\abs{\ip{P((f\cdot\grad) g)}{h}}
    =\abs{\ip{(f\cdot\grad)g}{P h}}
    \leq C \norm{f}_{H^1}\norm{\grad g}_{H^{\frac{1}{2}} }\norm{h}_{L^2},
$$
proving~\eqref{abcnonL2}.

For~\eqref{abcnonH1}, we first integrate by parts and observe $\lap P$ is a regular differential operator (identity~\eqref{eqnGradDivId}). Now we can integrate by parts again to obtain the desired estimate. Explicitly,
\begin{multline*}
\ip{\grad P((f \cdot \grad)g)}{\grad h}
 = -\ip{\lap P((f \cdot \grad)g)}{h} 
\\ 
= -\ip{\left(\lap - \grad \divergence \right) ((f \cdot \grad)g)}{h}
= -\ip{(f \cdot \grad)g}{ \left(\lap - \grad \divergence\right) h} ,
\end{multline*}
where all boundary integrals vanish because $f, g, h \in H^1_0$.
Now using~\eqref{eqnIntFGH} with~$s_1 = 1$, $s_2 = 1/2$, $s_3 = 0$, and elliptic regularity we have
$$
\abs*{\ip{\grad P((f \cdot \grad)g)}{\grad h}}
    = \abs*{\ip{(f \cdot \grad)g}{ \left(\lap - \grad \divergence\right) h}}
    \leq C \norm{f}_{H^1 } \norm{\grad g}_{H^{\frac{1}{2}} } \norm{\lap h}_{L^2}.
$$
This concludes the proof.
\end{proof}

We now return to the proof of Theorem~\ref{thmSmallDataGexist}.
\begin{proof}
We assume there exists a smooth solution  $u$ of~\eqsysENS on the time interval $[0, T]$ for some $T > 0$. We will find appropriate {\it a priori} estimates for the norm of $u$ (see relation~\eqref{eqnUspaces}, below) in terms of the initial data and $T$. Now a standard approximating scheme (e.g. the one constructed in~\cite{bblLiuLiuPego}) will prove global existence of solutions.

Fix $\epsilon > 0$ to be small enough so that Proposition~\ref{ppnAdjustedIP} holds, and $\ip{\cdot}{\cdot}_\epsilon$ denote the $H^1$ equivalent inner product from Proposition~\ref{ppnAdjustedIP}. Then
\begin{equation}\label{eqnH1EpsNormEvol}
\frac{1}{2} \del_t \norm{u}_{H^1_\epsilon}^2 + \ip{P ((u \cdot \grad) u)}{u}_\epsilon + \ip{u}{Au}_\epsilon = 0.
\end{equation}
where
$$
\norm{v}_{H^1_\epsilon} \defeq \sqrt{\ip{v}{v}_\epsilon}.
$$

By Lemma~\ref{lmaNonGenEst}, for any $c_0 > 0$, we can find a constant $C = C(\epsilon, c_0, \Omega) > 0$ such that
\begin{multline}\label{eqnNLStupidBound}
\abs*{\ip{P((u \cdot \grad) u)}{u}_\epsilon}
    \leq \norm{u}_{H^{1}} \norm{\grad u}_{H^{1/2}} \norm{\lap u}_{L^2}\\
    \leq C \norm{\grad u}_{L^2}^{3/2} \norm{\lap u}_{L^2}^{3/2}
    \leq C \norm{\grad u}^6 + \frac{\epsilon}{8 c_0} \norm{\lap u}^2.
\end{multline}
We will subsequently fix $c_0$ to be the constant $c$ that appears on the right of~\eqref{eqnCoercivicity}.

Using Proposition~\ref{ppnAdjustedIP} and equations~\eqref{eqnH1EpsNormEvol}, \eqref{eqnNLStupidBound} we obtain
$$
\del_t \norm{u}_{H^1_\epsilon}^2 + \frac{2}{c_0} \left( \norm{\grad u}_{L^2}^2 + \frac{\epsilon}{2} \norm{\lap u}_{L^2}^2 + C_\epsilon \norm{\grad q(u)}_{L^2}^2 \right) \leq C \norm{\grad u}_{L^2}^6,
$$
where $C_\epsilon$ is the constant in~\eqref{eqnCoercivicity}. Allowing the constant $C = C(\epsilon, c_0, \Omega)$ to change from line to line, and using the Poincar\'e inequality, we obtain
 $$
\del_t \norm{u}_{H^1_\epsilon}^2 + \frac{1}{c_0} \left( \norm{\grad u}_{L^2}^2 + \frac{\epsilon}{2} \norm{\lap u}_{L^2}^2
+ C_\epsilon \norm{\grad q(u)}_{L^2}^2 \right) \leq C \norm{u}_{H^1_\epsilon}^6 - \frac{1}{c_1} \norm{u}_{H^1_\epsilon}^2.
$$
for some constant $c_1 = c_1(\epsilon, \Omega)$. Thus if at time $t = 0$ we have
$$
\norm{u_0}_{H^1_\epsilon} \leq \frac{1}{( C c_1 )^{1/4}},
$$
then for all $t > 0$,
$$
\norm{u(t)}_{H^1_\epsilon}^2 + \frac{1}{c_0} \int_0^t \left( \norm{\grad u}_{L^2}^2 + \frac{\epsilon}{2} \norm{\lap u}_{L^2}^2 + C_\epsilon \norm{\grad q(u)}_{L^2}^2 \right) \, ds \leq \norm{u_0}_{H^1_\epsilon}^2.
$$
Now using the local existence result in~\cite{bblLiuLiuPego}, and the fact that $\norm{\cdot}_{H^1_\epsilon}$ is equivalent to the usual $H^1$ norm, 
we conclude the proof of Theorem~\ref{thmSmallDataGexist}.
\end{proof}

\section{Two dimensional Small divergence global existence for the extended Navier-Stokes equations}\label{sxn2DSmallDivGlobalExistence}

The aim of this section is to prove Theorem~\ref{thmSmallDivGexist}. We recall first the $H^1_0$-orthogonal projection onto divergence free vector fields. For $u_0 \in H^1_0(\Omega)$, we define $v_0, w_0 \in H^1_0(\Omega)$ to be solutions of the PDE's
\begin{equation}\label{eqnV0andW0}
\begin{beqn}
- \lap v_0 + \grad \phi = -\lap u_0 & in $\Omega$,\\
\divergence v_0 = 0 & in $\Omega$,\\
v = 0 & on $\del \Omega$,
\end{beqn}
\quad\text{and}\quad
\begin{beqn}
- \lap w_0 + \grad \psi = 0 & in $\Omega$,\\
\divergence w_0 = \divergence u_0 & in $\Omega$,\\
w_0 = 0 & on $\del \Omega$.
\end{beqn}
\end{equation}
respectively. Note that $u_0 \in H^1_0(\Omega)$ guarantees the required compatibility condition $\int_\Omega \divergence u_0 = 0$, and so the existence of $v_0, w_0$ satisfying~\eqref{eqnV0andW0} is well known (see for instance~\cite{bblTemam}*{\S2}). Clearly $u_0 = v_0 + w_0$, and orthogonality of $v_0$ and $w_0$ in $H^1_0$ follows from the identity
$$
\ip{v_0}{w_0}_{H^1_0(\Omega)} \defeq \ip{\grad v_0}{\grad w_0} = \ip{v_0}{-\lap w_0} = \ip{v_0}{-\grad \psi} = 0.
$$

Let $v$ be the solution to equation~\eqsysENS with initial data $v_0$. As shown earlier, $\divergence v_0 = 0$ implies that $\divergence v = 0$ for all time, and consequently $v$ is a solution of the $2D$ Navier-Stokes  with initial data $v_0$. Let $w = u - v$, and observe
\begin{equation}\label{eqnW}
\left\{
\begin{aligned}
\begin{multlined}[b]
\del_t w + P((w \cdot \grad) w) \mathop+\\
\mathop+ P((v\cdot \grad) w ) + P((w \cdot \grad) v)
\end{multlined}
&= \lap w -\grad p_s(w) &&\text{in } \Omega, \\
w &= 0 &&\text{on } \partial\Omega.\\
\end{aligned}\right.
\end{equation}

The strategy to prove Theorem~\ref{thmSmallDivGexist} is as follows. First standard existence theory for the  $2D$ Navier-Stokes equations implies that for any initial data $v_0\in H^1_0$, with $\divergence v_0=0$, we have global existence of a strong solution $v$. Further, after a long time $T_0$, the solution $v$ becomes small. Now making $u_0 - v_0$ is sufficiently small, we can guarantee that $w$, a solution to~\eqref{eqnW} with initial data $u_0 - v_0$, both exists on the time interval $[0, T_0]$, and is small at time $T_0$. Thus $u = v+ w$ is a solution to~\eqsysENS defined, which is small at time $T_0$. Now a small data global existence result (Theorem~\ref{thmSmallDataGexist}) will allow us to continue this solution for all time.

We begin with a Lemma concerning the existence and smallness of solutions to~\eqref{eqnW}.
\begin{lemma}\label{lmaEqnWsmalltime}
Let $\Omega \subset \R^2$ be a bounded $C^3$ domain, and $v_0 \in
H^1_0(\Omega)$ with $\divergence v_0 = 0$. Let $u_0\in H^1_0(\Omega)$ be
such that $P_0u_0=v_0$. Then, for any $T_0,\delta_0>0$ there exists a (small) constant $W_0=W_0(\Omega,\|v_0\|_{H^1},T_0,\delta_0)$ such that  if
$$
\norm{w_0}_{H^1_\epsilon}\leq W_0
$$
then there exists a solution of \eqref{eqnW} on the interval $[0,T_0]$ and
$$
\norm{w(T_0)}_{H^1}\leq \delta_0.
$$
\end{lemma}

Momentarily postponing the proof of the lemma, we prove Theorem~\ref{thmSmallDivGexist}.
\begin{proof}[Proof of Theorem~\ref{thmSmallDivGexist}]
We let $V_0$ be as in Theorem~\ref{thmSmallDataGexist}, and let $v$ be
the solution to the 2D Navier-Stokes equations with initial data $v_0 =
P_0 u_0 \in H^1_0$. It is well known (see for
instance~\cites{bblConstFoias,bblTemam}) that there exists $T_0$ large
enough, so that $\norm{v(T_0)}_{H^1_0}\leq \frac12 {V_0}$.
Indeed, from the standard $L^2$ energy identity we can choose $T_0$ to
satisfy $T_0(\frac12 V_0)^2\le \|v_0\|_{L^2}^2$.

By Lemma~\ref{lmaEqnWsmalltime} there exists $W_0 > 0$ small enough so that if initially
\begin{equation}\label{eqnW0Small}
\norm{w_0}_{H^1_\epsilon}\leq W_0
\end{equation}
then the solution $w$ to~\eqref{eqnW} exists up to time $T_0$, and
further $\norm{w(T_0)}_{H^1}\leq \frac12V_0$. From~\eqref{eqnV0andW0}, we know $\norm{w_0}_{H^1} \leq c \norm{\divergence u_0}_{L^2}$ (see~\cite{bblTemam}*{\S2}). Since the norms $\norm{\cdot}_{H^1_\epsilon}$ and $\norm{\cdot}_{H^1}$ are equivalent, making $U_0$ small enough will guarantee~\eqref{eqnW0Small}, thus allowing to apply Lemma~\ref{lmaEqnWsmalltime} and obtain the existence of $w$ on the interval $[0,T_0]$.
 Then we obtain that ~\eqsysENS has a solution $u=w+v$ on $[0,T_0]$  and moreover $\norm{u(T_0)}_{H^1}\leq \norm{w(T_0)}_{H^1}+\norm{v(T_0)}_{H^1}\leq V_0$. Applying Theorem~\ref{thmSmallDataGexist} we can continue the solution $u$ on the interval $[T_0,\infty)$. 
\end{proof}

It remains to prove the Lemma.
\begin{proof}[Proof of Lemma~\ref{lmaEqnWsmalltime}]
As with the proof of Theorem~\ref{thmSmallDataGexist}, it suffices to obtain an {\it a priori} estimate for~$\norm{w}_{H^1}$. Fix $\epsilon > 0$ to be small enough so that Proposition~\ref{ppnAdjustedIP} holds. Then
\begin{multline}\label{eqnGradWextended}
\frac{1}{2}\del_t \norm{w}_{H^1_\epsilon}^2 +\overbrace{\ip{P((w\cdot\grad) w)}{w}_\epsilon+\ip{w}{Aw}_\epsilon}^{\mathcal{J}_1}=\\
 -\underbrace{\ip{ P ((v \cdot \grad) w)}{w}_\epsilon}_{\mathcal J_2} -
\underbrace{\ip{P((w \cdot \grad) v)}{w}_\epsilon}_{\mathcal J_3}
\end{multline}
where $\ip{\cdot}{\cdot}_\epsilon$ denotes the inner product defined in Proposition~\ref{ppnAdjustedIP}, and $\norm{\cdot}_{H^1_\epsilon}$ the induced norm.

We estimate each term individually. The term $\mathcal J_1$ is identical to the term that appears in the proof of Theorem~\ref{thmSmallDataGexist}, and thus
$$
\mathcal J_1 \geq \frac{1}{c_0} \left( \norm{\grad w}_{L^2}^2 + \frac{\epsilon}{2} \norm{\lap w}_{L^2}^2 + C_\epsilon \norm{\grad q}_{L^2}^2 \right) - \left( c_2 \norm{w}_{H^1_\epsilon}^6 - \frac{1}{c_1} \norm{w}_{H^1_\epsilon}^2 \right),
$$
where $q$ is the unique, mean-zero function such that $\grad q = (I - P) w$. As before, $c_0$ is the constant that appears in~\eqref{eqnCoercivicity}, and $c_1 = c_1(\epsilon, \Omega), c_2 = c_2(\epsilon, \Omega)$ are positive constants.

Using $C = C(\epsilon, \Omega) > 0$ to denote an intermediate constant that can change from line to line, Lemma~\ref{lmaNonGenEst} bounds $\mathcal J_2$ and $\mathcal J_3$ by
\begin{align*}
\abs{\mathcal{J}_2} + \abs{\mathcal J_3}
    &\leq C \left(\norm{v}_{H^{1}} \norm{\grad w}_{H^{1/2}} \norm{\lap w}_{L^2} + \norm{w}_{H^{1}} \norm{\grad v}_{H^{1/2}} \norm{\lap w}_{L^2} \right)\\
    &\leq \frac{\epsilon}{4c_0}\norm{\lap w}_{L^2}^2 + c_3 \left( \norm{v}_{H^1}^4 + \norm{\grad v}_{H^{1/2} }^2 \right)\norm{w}_{H^1_\epsilon}^2
\end{align*}
for some constant $c_3 = c_3(\epsilon, \Omega)$.

Combining our estimates,
\begin{multline}\label{eqnGraeqnExtendedDynamicsdW1}
\frac{1}{2}\del_t  \norm{w}_{H^1_\epsilon}^2 +\frac{1}{c_0} \left( \norm{\grad w}_{L^2}^2 + \frac{\epsilon}{4} \norm{\lap w}_{L^2}^2 + C_\epsilon \norm{\grad q}_{L^2}^2 \right)\\
    \leq c_2 \norm{w}_{H^1_\epsilon}^6 - \left(\frac{1}{c_1} - c_3 \left( \norm{v}_{H^1}^4 + \norm{\grad v}_{H^{1/2} }^2 \right) \right) \norm{w}_{H^1_\epsilon}^2.
\end{multline}

Since $v$ is a strong solution to the 2D incompressible Navier-Stokes equations with initial data $v_0\in H^1_0$,  we have (see for instance \cite{bblConstFoias}*{p.\ 78}) that 
$$
\sup_{t \geq 0} \norm{v(t)}_{H^1}^2 + \int_0^\infty \norm{v(s)}_{H^{1}}^2 \, ds < C
$$
for some constant $C$ depending only on $\Omega$ and $\norm{v_0}_{H^{1}}$. Using this in~\eqref{eqnGraeqnExtendedDynamicsdW1} will prove local well-posedness of~\eqref{eqnW}. Further, for any $T_0,\delta_0>0$, equation~\eqref{eqnGraeqnExtendedDynamicsdW1} will also show that the solution to~\eqref{eqnW} exists up to time $T_0$, and $\norm{w(T_0)}_{H^{1}} < \delta_0$, provided $\norm{w_0}_{H^{1}}$ is small enough.
\end{proof}

\section{Divergence damped equations}\label{sxnDivergenceDamping}
The aim of this section is to prove coercivity of $B_\alpha$ (defined in~\eqref{eqnBAlphaDef}), with constants independent of $\alpha$, and 2D global existence with strong enough divergence damping (Corollary~\ref{clyDivergenceDampedExistence}).

\begin{proposition}\label{ppnBCoercive}
For any $\alpha \geq 0$, and $u \in D(B_\alpha)$ we have
\begin{equation}\label{eqnBCoercive1}
\ip{u}{B_\alpha u}_\epsilon = \ip{u}{Au}_\epsilon + \alpha \left( \norm{\grad Q(u)}_{L^2}^2 + C_\epsilon \norm{Q(u)}_{L^2}^2 + \epsilon \norm{\lap Q(u)}_{L^2}^2 \right).
\end{equation}
\end{proposition}

\begin{proof}
By linearity,
\begin{equation}\label{eqnIPUBalphaEps}
\ip{u}{B_\alpha u}_\epsilon
    = \ip{u}{A u}_\epsilon + \alpha \ip{u}{(I - P) u}_\epsilon.
\end{equation}
For the second term on the right,
$$
\ip{u}{(I-P)u}_\epsilon
    = \ip{u}{(I-P)u} + C_\epsilon \ip{Q(u)}{Q((I-P)u)} + \epsilon \ip{\grad u}{\grad (I - P) u}.
$$
The first two terms on the right are equal to $\norm{(I-P)u}_{L^2}^2$ and $C_\epsilon \norm{Q(u)}_{L^2}^2$ respectively. For the last term,
\begin{align*}
\ip{\grad u}{\grad (I - P) u}
    &= - \ip{u}{\lap (I- P) u} + \int_{\del \Omega} u_i \frac{\del}{\del \nu} \left( (I - P) u \right)_i\\
    &= - \ip{u}{\grad \divergence u} + 0\\
    & = \norm{\divergence u}_{L^2}^2 - \int_{\del \Omega} (\divergence u) u \cdot \nu = \norm{\divergence u}_{L^2}^2
\end{align*}
Consequently,
\begin{equation}\label{eqnIPuImPu}
\ip{u}{(I- P)u}_\epsilon = \norm{(I - P) u}_{L^2}^2 + C_\epsilon \norm{Q(u)}_{L^2}^2 + \epsilon \norm{\divergence u}_{L^2}^2,
\end{equation}
and using~\eqref{eqnIPUBalphaEps}, we obtain~\eqref{eqnBCoercive1}.
\end{proof}

Before moving to the proof of Corollary~\ref{clyDivergenceDampedExistence}, we digress briefly to remark that we can also consider higher order divergence damped operators of the form
$$
B_\alpha' \defeq A  - \alpha \grad \divergence.
$$
The results we obtain for~\eqref{eqnNS}--\eqref{eqnPdef} with a zeroth order damping term will also apply when we add the second order damping term above. However, while the operator $B_\alpha'$ has a stronger (second order) damping term, it is not as easy to deal with numerically. The zeroth order damping terms in $B_\alpha$, on the other hand, can easily be implemented numerically, and has a strong enough damping effect to give a better existence result (Corollary~\ref{clyDivergenceDampedExistence}). We now return to prove Corollary~\ref{clyDivergenceDampedExistence}.

\begin{proof}[Proof of Corollary~\ref{clyDivergenceDampedExistence}]
Since Theorems~\ref{thmSmallDataGexist} and~\ref{thmSmallDivGexist} work verbatim for~\eqref{eqnExtendedDynamicsAlpha}, there exists a time $T_0 = T_0( \norm{u_0}_{H^{1}}, \Omega)$, independent of $\alpha$, such that  there exists a solution $u$ of \eqref{eqnExtendedDynamicsAlpha}  on the interval  $[0,T_0]$, with $\norm{u(T_0)}_{H^1}$ bounded, independent of $\alpha$. Let $U_0$ be the constant from Theorem~\ref{thmSmallDivGexist}. Observe that~\eqref{eqnExtendedDynamicsAlpha} implies that $\divergence u$ satisfies
\begin{equation*}
\left\{
\begin{aligned}
\del_t \divergence u + \alpha \divergence u &= \lap \divergence u &&\text{in }\Omega,\\
\frac{\del}{\del \nu} \divergence u &= 0 &&\text{for }x \in \del \Omega, t > 0,
\end{aligned}\right.
\end{equation*}
Consequently,
$$
\norm{\divergence u(t)}_{L^2}^2 \leq e^{-(\lambda_1 + \alpha) t} \norm{\divergence u_0}_{L^2}^2,
$$
where $\lambda_1 > 0$ is the smallest non-zero eigenvalue of the Laplacian with Neumann boundary conditions. Thus there exists $\alpha_0 > 0$, such that
$$
\norm{\divergence u(T_0)}_{L^2} < U_0,
$$
for all $\alpha > \alpha_0$. Now, by Theorem~\ref{thmSmallDivGexist} the solution to~\eqref{eqnExtendedDynamicsAlpha} also exists and is regular on the time interval $[T_0, \infty)$.
\end{proof}

\section{Existence results under coercive boundary conditions.}\label{sxnBC}
The aim of this section is to show that the extended Stokes operator is coercive under the boundary conditions~\eqref{eqnNlBC}, and prove Proposition~\ref{ppnNLBCGexist}. We begin with coercivity.

\begin{proposition}\label{ppnCoercivity}
If either $u$ and $v$ are in $H^1(\T^d)$ and periodic, or if $u$ and $v$ are in $H^2(\Omega)$ and satisfy the boundary conditions~\eqref{eqnNlBC}, then
\begin{equation}\label{eqnAselfAdj}
\ip{Au}{v} = \ip{u}{Av}
\quad\text{and}\quad
\ip{u}{Au} = \int_\Omega \abs{\grad u}^2.
\end{equation}
\end{proposition}
\begin{proof}
In the periodic case, $P\lap = \lap P$. Thus $\grad p_s = 0$, $A = -\lap$, and both equalities in~\eqref{eqnAselfAdj} follow easily.

Suppose now $u, v$ satisfy~\eqref{eqnNlBC}. In view of~\eqref{eqnAdef}, we have
\begin{multline*}
\ip{Au}{v} = -\ip{P\lap u}{v} - \ip{\grad \divergence u}{v} = -\ip{P\lap P u}{v} - \ip{\grad \divergence u}{v} \\
= \ip{\grad P u}{\grad P v} + \ip{\divergence u}{\divergence v} - \int_{\del\Omega} \left[ (Pv)_i \frac{\del (Pu)_i}{\del \nu} - (\divergence u) v \cdot \nu \right].
\end{multline*}
Observe that $Pv = 0$ on $\del\Omega$, because because $Pv \cdot \nu = 0$ by definition of $P$, and $Pv \cdot \tau = 0$ by~\eqref{eqnNlBC}. Thus both the above boundary integrals vanish, giving
$$
\ip{Au}{v} = \ip{\grad P u}{\grad P v} + \ip{\divergence u}{\divergence v}.
$$
A similar calculation shows
$$
\ip{Av}{u} = \ip{\grad Pu}{\grad Pv} + \ip{\divergence u}{\divergence v}
$$
proving that $A$ is self adjoint.

Now because $Pu = Pv = 0$ on $\del \Omega$, a direct calculation shows that
$$
\ip{\grad Pu}{\grad Pv} = \ip{\curl Pu}{\curl Pv} = \ip{\curl u}{\curl v}.
$$
Consequently, we see
$$
\ip{Au}{v} = \ip{\curl u}{\curl v} + \ip{\divergence u}{\divergence v} = \ip{\grad u}{\grad v}.
$$
Setting $u = v$, the second assertion in~\eqref{eqnAselfAdj} follows.
\end{proof}

Finally, we turn to Proposition~\ref{ppnNLBCGexist}. Before presenting the proof, we remark that if we instead impose periodic boundary conditions, Proposition~\ref{ppnNLBCGexist} and its proof (below) go through almost unchanged. The only modification required is the justification of the Poincar\'e inequality that will be (implicitly) used in many estimates. For this justification, observe that with periodic boundary conditions, the mean of solutions to~\eqref{eqnNS}--\eqref{eqnPdef} is conserved. Thus, by switching to a moving frame, we can assume that the initial data, and hence the solution for all time, are mean zero. This will justify the use of the Poincar\'e inequality in the proof. With this, we prove Proposition~\ref{ppnNLBCGexist}.

\begin{proof}[Proof of Proposition~\ref{ppnNLBCGexist}]
Let $v = Pu$, and $q = Q(u)$.
Since
\[
P((\grad q\cdot\grad)\grad q)= P(\grad |\grad q|^2/2)=0,
\]
applying $P$ to~\eqref{eqnNS} gives
\begin{equation}\label{per.v}
\left\{
\begin{IEEEeqnarraybox}[][c]{c?s}
\del_t v - P \lap v + P((v\cdot \grad)(v+ \grad q)+(\grad q\cdot\grad)v) =0 & in $\Omega$,\\
v = 0 & on $\del \Omega$,%
\end{IEEEeqnarraybox}\right.
\end{equation}
where the boundary condition on $v$ comes from~\eqref{eqnNlBC}. The point is that
energy estimates can be used directly to estimate $v$,
since it satisfies explicit boundary conditions.

Since $Pv = v$, multiplying \eqref{per.v} by $v$ and integrating yields,
\begin{align}
\nonumber
\frac12\del_t  \norm{v}_{L^2}^2 + \norm{\grad v}_{L^2}^2
    &= \frac12 \int_\Omega |v|^2 \lap q - \int_\Omega v\cdot((v\cdot\grad)\grad q)\\
    \nonumber
    &\leq C \norm{v}_{L^4}^2 \norm{\grad^2 q}_{L^2} \leq C \norm{v}_{L^2} \norm{\grad v}_{L^2} \norm{\lap q}_{L^2}\\
    \label{per.qe}
    &\leq \frac12 \norm{\grad v}_{L^2}^2 + C \norm{v}_{L^2}^2 \norm{\divergence u}_{L^2}^2.
\end{align}
Here we used elliptic regularity to control $\norm{\grad^2 q}$ by $\norm{\lap q}$, which is valid since $\frac{\del q}{\del \nu} = 0$ on $\del \Omega$. We also used the (2D) Ladyzhenskaya inequality $\norm{v}_{L^4}^2 \leq C \norm{v}_{L^2} \norm{\grad v}_{L^2}$, which is valid since $v = 0$ on $\del \Omega$.

Since $\divergence u$ is a mean-zero solution of~\eqref{eqnHeat}, we know that
$$
\int_0^\infty \norm{\divergence u(t)}_{L^2}^2 \, dt \leq
    \frac{1}{2 \lambda_1} \norm{\divergence u_0}_{L^2}^2,
$$
where $\lambda_1$ is the smallest non-zero eigenvalue of the Laplacian with Neumann boundary conditions. Thus Gronwall's lemma and~\eqref{per.qe} gives the closed estimate
\begin{equation}\label{eqnNLBCEE}
\norm{v(t)}_{L^2}^2 + \int_0^t \norm{\grad v(s)}_{L^2}^2 \,ds \leq \exp\left(C \norm{\divergence u_0}_{L^2}^2 \right) \norm{v_0}_{L^2}^2.
\end{equation}

Since $Pv = v$, regularity of the (standard) Stokes operator tells us that the norms
$\norm{-P\lap v}_{L^2}$ and  $\norm{v}_{H^{2}}$ are equivalent (see for instance~\cite{bblConstFoias}*{Chapter 4}). Multiplying \eqref{per.v} by $-P \lap v$, integrating by parts, and using~\eqref{eqnIntFGH} to bound the nonlinear term in the usual way gives
$$
\del_t \norm{\grad v}_{L^2}^2 + \frac{1}{c} \norm{\lap v}_{L^2}^2 \leq C\left( \norm{v}_{L^2}^2 \norm{\grad v}_{L^2}^2 + \norm{\grad \divergence u}_{L^2}^2 + \norm{\divergence u}_{L^2}^4 \right) \norm{\grad v}_{L^2}^2.
$$
Using Gronwall's lemma, equation~\eqref{eqnNLBCEE} and~\eqref{eqnHeat}, we obtain
\begin{equation}\label{eqnNLBCH1}
\norm{\grad v(t)}_{L^2}^2 + \frac{1}{c} \int_0^t \norm{\lap v(s)}_{L^2}^2 \, ds \leq K \norm{\grad v_0}_{L^2}^2,
\end{equation}
for some constant $K = K(\Omega, \norm{\divergence u_0}_{L^2}, \norm{v_0}_{L^2})$. In fact, one can bound $K$ above by
$$
K \leq C \exp\left(C\left( \exp\left(C \norm{\divergence u_0}_{L^2}^2 \right) \norm{v_0}_{L^2}^4 + \norm{\divergence u_0}_{L^2}^2 + \norm{\divergence u_0}_{L^2}^4 \right) \right)
$$
for some constant $C = C(\Omega)$.

Finally, we consider a Galerkian scheme for~\eqref{eqnNS}--\eqref{eqnPdef} using eigenfunctions of the Stokes operator (with no-slip boundary conditions), and gradients of eigenfunctions of the Laplacian (with no-flux boundary conditions). It is easy to check that these Galerkian approximations satisfy the same energy estimates~\eqref{eqnNLBCEE} and~\eqref{eqnNLBCH1}. A bound for $\del_t u$ will then follow from~\eqref{eqnNS}, and standard techniques will prove global existence.
\end{proof}

\appendix

\section{Failure of coercivity under the standard inner product.}\label{sxnNonPositivity}

Most of this section is devoted to the proof that Stokes operator is not positive under the standard $L^2$ inner product (Proposition~\ref{ppnNonPositivity}).

\begin{proof}[Proof of Proposition~\ref{ppnNonPositivity}]
As mentioned earlier, the key idea in the proof is to identify the harmonic conjugate of the Stokes pressure as the harmonic extension of the vorticity. We begin by working up to this. Since
$$
\lap p_s = \divergence \grad p_s = \divergence \left(\lap  P - P \lap\right) u = 0,
$$
the Poincar\'e lemma guarantees the existence of $q_s$ such that
\begin{equation}\label{eqnQdef}
\grad p_s = \gradperp q_s \defeq
\begin{pmatrix}
-\del_2 q_s\\
\phantom{-}\del_1 q_s
\end{pmatrix}.
\end{equation}
Observe that both $p_s$ and $q_s$ are harmonic. Indeed,
\begin{equation}\label{eqnPQHarmonic}
\lap q_s = \curl \gradperp q_s = \curl \grad p_s = 0.
\end{equation}
We remark that equations~\eqref{eqnQdef} and~\eqref{eqnPQHarmonic} above show that $-q_s$ is the harmonic conjugate of $p_s$.

To obtain boundary conditions for $q_s$, let $\tau = -\nu^\perp$ be the
unit tangent vector on $\del \Omega$. To clarify our sign convention, if
$\nu = \binom{\nu_1}{\nu_2}$, then $\tau \defeq \binom{\nu_2}{-\nu_1}$.
Now observe
\begin{multline*}
\frac{\del q_s}{\del \tau}
    = \grad q_s \cdot \tau
    = \gradperp q_s \cdot \nu
    = \frac{\del p_s}{\del \nu}
    = \nu \cdot \left(\lap P - P \lap\right) u\\
    = \nu \cdot (\lap u - \grad \divergence u)
    = \nu \cdot \gradperp \curl u
    = \tau \cdot \grad \omega
    = \frac{\del \omega}{\del \tau},
\end{multline*}
where, as before,  $\curl u = \del_1 u_2 -\del_2 u_1$ is the two dimensional curl, and $\omega = \curl u$. Thus, adding a constant to $q_s$, we may, without loss of generality assume
\begin{equation}
q_s = \omega \quad \text{on }\del\Omega.
\end{equation}

A direct calculation shows
\begin{equation*}
\int_\Omega u \cdot A u
    = \int_\Omega \abs{\grad u}^2 + \int_\Omega \grad p_s \cdot u
    = \int_\Omega \left( \omega^2 + \abs{\divergence u}^2 \right) - \int_\Omega q_s \cdot \omega,
\end{equation*}
where we used the boundary condition $u = 0$ on $\del \Omega$ to integrate by parts. Thus to prove Proposition~\ref{ppnNonPositivity}, it is enough to produce a function $u$, satisfying the required boundary conditions, such that
\begin{equation}\label{CoercivityContradiction}
\int_\Omega q_s \omega \geq \norm{\omega}_{L^2}^2 + (C + 1) \norm{\divergence u}_{L^2}^2
\end{equation}
We prove the existence of such functions separately.
\begin{lemma}\label{lmaCoercivityContradiction}
For any $C > 0$, there exists $u\in H^1_0(\Omega)$ such that~\eqref{CoercivityContradiction} holds. As usual, $\omega=\curl u$,  and $q_s$ is the solution of the Dirichlet problem
\begin{equation}\label{eqnQ}
\left\{
\begin{aligned}
\lap q_s &= 0 &&\text{in }\Omega,\\
q_s &= \omega &&\text{on }\del \Omega
\end{aligned}\right.
\end{equation}
\end{lemma}
The Lemma immediately finishes the proof of Proposition~\ref{ppnNonPositivity}.
\end{proof}

\begin{proof}[Proof of Lemma~\ref{lmaCoercivityContradiction}.]
We look for $u$ of the form $u=v+\nabla p$ with $\divergence v=0$ in $\Omega$ and $v\cdot \nu=0$ on $\partial\Omega$, where $\nu$ denotes the outward pointing normal vector on the boundary. Then there exists $\psi$ on $\Omega$ so that $v=\gradperp\psi$. The boundary condition $v\cdot \nu=0$ becomes $\frac{\partial\psi}{\partial\tau}=0$ where $\tau$ denotes the tangential direction on $\partial\Omega$.

We note that $u=0$ and $v\cdot\nu=0$ on $\partial\Omega$ imply
$$
 \frac{\partial p}{\partial \tau}=-v\cdot\tau=-\frac{\partial\psi}{\partial \nu},\quad \frac{\partial p}{\partial \nu}=0,\quad v\cdot\nu=\frac{\partial\psi}{\partial\tau}=0.
$$
As $\frac{\partial\psi}{\partial\tau}=0$ and $\psi$ is determined up to a constant we can assume without loss of generality that $\psi=0$ on $\partial\Omega$ and then for a given $u$ the stream function $\psi$ is uniquely determined as the solution of the Dirichlet problem
\begin{alignat*}{2}
\lap \psi &= \omega \quad&& \text{in } \Omega,\\
\psi &= 0 && \text{on }\del \Omega.
\end{alignat*}
Then, we have
\begin{align*}
\int_\Omega q_s\omega\,dx=\int_\Omega q_s\lap\psi\,dx=\int_\Omega \lap q_s\psi\,dx+\int_{\partial\Omega} \frac{\partial\psi}{\partial\nu}q_s\,d\sigma-\int_{\partial\Omega}\frac{\partial q_s}{\partial\nu}\psi\,d\sigma\\
=\int_{\partial\Omega} \frac{\partial\psi}{\partial\nu}q_s\,d\sigma=\int_{\partial\Omega}\frac{\partial\psi}{\partial\nu}\lap\psi\,d\sigma,
\end{align*}
and~\eqref{CoercivityContradiction} becomes
\begin{equation}\label{eqnCoercivityContradictionPlus}
\int_{\partial\Omega}\frac{\partial\psi}{\partial\nu}\lap\psi\,d\sigma\geq\norm{\lap \psi}_{L^2}^2+(C+1)\|\lap p\|_{L^2}^2.
\end{equation}

Summarizing it suffices to find $\psi,p$ such that \eqref{eqnCoercivityContradictionPlus} holds together with the boundary conditions
\begin{equation}\label{PsiPCompatibility}
\psi=0,
\quad \frac{\partial p}{\partial \nu}=0,
\quad\text{and}\quad \frac{\partial p}{\partial \tau}=-\frac{\partial\psi}{\partial\nu}
\quad\text{on } \del \Omega.
\end{equation}
Fix some point $x_0\in\partial\Omega$ and let 
$s\mapsto \hat x(s)$ be an arclength parametrization
of the ($C^3$) boundary $\partial\Omega$, such that $x_0=\hat x(0)$
and  oriented so that the outward unit normal 
$\hat\nu(s)$ at $\hat x(s)$ satisfies $\hat\nu(s)^\perp=\hat x'(s)$.
Then the map $(s,r)\mapsto x=\hat x(s)-r\hat\nu(s)$ is $C^2$ and 
is locally invertible near $x_0$, providing 
orthogonal coordinates 
$x\mapsto (s,r) \in(-\epsilon,\epsilon)\times(0,\epsilon)$
in some neighborhood of $x_0$ in $\Omega$.

We fix $p$ to be of the form $p(x)=\alpha(s)\beta(r)$ where
$\alpha$ and $\beta$ are in 
$C^\infty_c((-\epsilon,\epsilon))$
and $\beta(0)=1$, $\beta'(0)=0$.  
We will then choose $\psi$ of the form
$\psi(x)=\alpha'(s)\gamma(r)$ where 
$\gamma\in C^\infty_c((-\epsilon,\epsilon))$
with $\gamma(0)=0$ and $\gamma'(0)=1$.
Then \eqref{PsiPCompatibility} will hold, and direct calculation shows
\[
\int_{\partial\Omega}\frac{\partial\psi}{\partial\nu}\lap\psi\,d\sigma
= \int_{-\epsilon}^{\epsilon} \alpha'(s)^2(\gamma''(0)+\kappa(s))\,ds
\]
where $\kappa(s)=\lap r(\hat x(s))$ is the curvature of the boundary.
The right-hand side of \eqref{eqnCoercivityContradictionPlus}
on the other hand, is easily computed to be bounded by 
$C+C\|\gamma\|_{H^2}^2$, with a constant $C$ independent of the choice
of $\gamma$.  It is clear that $\gamma$ can be chosen to make 
$\gamma''(0)$ arbitrarily large
while $\|\gamma\|_{H^2}^2$ remains bounded. 
Thus \eqref{eqnCoercivityContradictionPlus} holds for some $\psi$ and $p$.
\end{proof}

\medskip Finally, to conclude this section we turn to the proof of Corollary~\ref{clyL2Increase}. Of course the proof is immediate from Proposition~\ref{ppnNonPositivity}, and we only present it here for completeness.
\begin{proof}[Proof of Corollary~\ref{clyL2Increase}]
Choose $u_0 \in C^2(\bar\Omega)$ to be such that~\eqref{eqnNonPositivity} holds, and let $u$ be the solution to~\eqref{eqnExtendedStokes} with initial data $u_0$. By continuity in time, we must have
$$
\int_\Omega u(t) \cdot A u(t) < 0
$$
for all $t$ in some small interval $[0, t_0]$. Thus $\del_t \norm{u}_{L^2}^2 = -\int_\Omega u \cdot Au > 0$ on the interval $(0, t_0]$ which immediately completes the proof.
\end{proof}

\section*{Acknowledgments}
The authors would like to thank James P. Kelliher for insightful discussions related to this work.
\end{document}